\documentclass[11pt,a4paper,
]{amsart}                 
\usepackage[T1]{fontenc}             
\usepackage[utf8]{inputenc}          \usepackage[english]{babel}       
\usepackage{graphicx}                      %
\usepackage[font=small]{quoting}            %
\usepackage{caption}  
\usepackage[colorlinks=true,linkcolor=blue]{hyperref}
\usepackage{amsmath}
\DeclareMathOperator{\Vol}{Vol}
\usepackage[top=3.5cm, bottom=4cm, right=2.5cm, left=2.5cm]{geometry}

\usepackage{verbatim}
\usepackage{color}
\usepackage{picture}
\usepackage{amsmath,amsthm,amsfonts,amssymb}
\usepackage{comment}
\usepackage[colorlinks=true]{hyperref}
\hypersetup{
	colorlinks=true,%
	citecolor=red,%
	filecolor=black,%
	linkcolor=blue,%
	urlcolor=blue
}

\allowdisplaybreaks

\newdimen\AAdi%
\newbox\AAbo%
\def\AAk#1#2{\s_etbox\AAbo=\hbox{#2}\AAdi=\wd\AAbo\kern#1\AAdi{}}%
\def\AAr#1#2#3{\s_etbox\AAbo=\hbox{#2}\AAdi=\ht\AAbo\raise#1\AAdi\hbox{#3}}%
\font\tenmsb=msbm10 at 12pt
\font\sevenmsb=msbm7 at 8pt
\font\fivemsb=msbm5 at 6pt
\newfam\msbfam
\textfont\msbfam=\tenmsb
\scriptfont\msbfam=\sevenmsb
\scriptscriptfont\msbfam=\fivemsb
\def\Bbb#1{{\tenmsb\fam\msbfam#1}}

\newtheorem{thm}{Theorem}[section] 
 
\newtheorem{lem}[thm]{Lemma} 
 
\newtheorem{cor}[thm]{Corollary} 
\newtheorem{corollary}[thm]{Corollary} 
\newtheorem{prop}[thm]{Proposition} 
\newtheorem{proposition}[thm]{Proposition}

    \newtheorem{rem}[thm]{Remark} 
		\newtheorem{remark}[thm]{Remark} 
\theoremstyle{remark}

\theoremstyle{definition}

\newcommand{\beq}{\begin{equation}}
\newcommand{\eeq}{\end{equation}}
\newcommand{\beqr}{\begin{eqnarray}}
\newcommand{\eeqr}{\end{eqnarray}}
\newcommand{\ba}{\begin{array}}
\newcommand{\ea}{\end{array}}

\begin{document}

\newtheorem{pro}{Proposition}
\newtheorem{defi}{Definition}
\newcommand{\noi}{\noindent}
\newcommand{\dis}{\displaystyle}
\newcommand{\mint}{-\!\!\!\!\!\!\int}

\def \bx{\hspace{2.5mm}\rule{2.5mm}{2.5mm}} \def \vs{\vspace*{0.2cm}}
\def\hs{\hspace*{0.6cm}}
\def \ds{\displaystyle}
\def \p{\partial}
\def \O{\Omega}
\def \o{\omega}
\def \b{\beta}
\def \m{\mu}
\def \l{\lambda}
\def\L{\Lambda}
\def \ul{u_\lambda}
\def \D{\Delta}
\def \d{\delta}
\def \k{\kappa}
\def \s{\sigma}
\def \e{\epsilon}
\def \a{\alpha}
\def \tf{\tilde{f}}
\def\cqfd{%
\mbox{ }%
\nolinebreak%
\hfill%
\rule{2mm} {2mm}%
\medbreak%
\par%
}
\def \pr {\noindent {\it Proof.} }
\def \rmk {\noindent {\it Remark} }
\def \esp {\hspace{4mm}}
\def \dsp {\hspace{2mm}}
\def \ssp {\hspace{1mm}}

\def \u{u_+^{p^*}}
\def \ui{(u_+)^{p^*+1}}
\def \ul{(u^k)_+^{p^*}}
\def \energy{\int_{\R^n}\u }
\def \sk{\s_k}
\def \mo{\mu_k}
\def\cal{\mathcal}
\def \I{{\cal I}}
\def \J{{\cal J}}
\def \K{{\cal K}}
\def \OM{\overline{M}} 
\def\CC {{\mathcal{C}}}

\def\fk{{{\cal F}}_k}
\def\M1{{{\cal M}}_1}
\def\Fk{{\cal F}_k}
\def\Fl{{\cal F}_l}
\def\FF{\cal F}
\def\Gk{{\Gamma_k^+}}
\def\n{\nabla}
\def\uuu{{\n ^2 u+du\otimes du-\frac {|\n u|^2} 2 g_0+S_{g_0}}}
\def\uuug{{\n ^2 u+du\otimes du-\frac {|\n u|^2} 2 g+S_{g}}}
\def\sku{\sk\left(\uuu\right)}
\def\qed{\cqfd}
\def\vvv{{\frac{\n ^2 v} v -\frac {|\n v|^2} {2v^2} g_0+S_{g_0}}}
\def\vvs{{\frac{\n ^2 \tilde v} {\tilde v}
 -\frac {|\n \tilde v|^2} {2\tilde v^2} g_{S^n}+S_{g_{S^n}}}}
\def\skv{\sk\left(\vvv\right)}
\def\tr{\hbox{tr}\,}
\def\pO{\partial \Omega}
\def\dist{\hbox{dist}}
\def\RR{\Bbb R}\def\R{\Bbb R}
\def\C{\Bbb C}
\def\B{\Bbb B}
\def\N{\Bbb N}
\def\Q{\Bbb Q}
\def\Z{\Bbb Z}
\def\PP{\Bbb P}
\def\EE{\Bbb E}
\def\F{\Bbb F}
\def\G{\Bbb G}
\def\H{\Bbb H}
\def\S{\mathbb{S}} 

\def\lcf{{locally conformally flat} }

\def\circledwedge{\setbox0=\hbox{$\bigcirc$}\relax \mathbin {\hbox
to0pt{\raise.5pt\hbox to\wd0{\hfil $\wedge$\hfil}\hss}\box0 }}

\def\sss{\frac{\s_2}{\s_1}}

\def\mathscr{\mathcal}

\newcommand{\eq}[1]{\begin{equation}\allowdisplaybreaks\begin{alignedat}{2} #1 \end{alignedat}\end{equation}}

\def\rho{\epsilon}

\numberwithin{equation} {section}

\date{}

\title{Optimal geometric  inequalities and fully nonlinear conformal flows }

\author{Yuxin Ge}
\address{Institut de Math\'ematiques de Toulouse,\\
Universit\'e Paul Sabatier,\\
118, route de Narbonne,\\
31062 Toulouse Cedex, France}
\email{yge@math.univ-toulouse.fr}

\author{Guofang Wang}
\address{ Albert-Ludwigs-Universit\"at Freiburg,
Mathematisches Institut,
Eckerstr. 1,
D-79104 Freiburg, Germany}
\email{guofang.wang@math.uni-freiburg.de}

\author{Wei Wei}
\address{School of Mathematics, Nanjing University, Nanjing 210093, P.R. China}
\email{wei\_wei@nju.edu.cn}

\begin{abstract}
We establish sharp Sobolev-type geometric inequalities on $\mathbb{S}^n$ involving the total $\sigma_k$-curvatures $\int_{\mathbb{S}^n}\sigma_k(g)\,dv_g$. These results extend the optimal inequalities of Guan--Wang~\cite{GWDuke} from the cone $\mathcal{C}_k$ to the strictly larger cone $\mathcal{C}_{k-1}$, thereby enlarging the range of admissible conformal metrics.
Our approach is variational and is implemented through a fully nonlinear conformal flow. Working in $\mathcal{C}_{k-1}$ introduces substantial analytic difficulties; in particular, one must obtain $C^2$ a priori estimates while simultaneously verifying that the flow remains parabolic. We resolve these issues via a carefully designed test function and by applying the maximum principle to the maximal eigenvalue of the Hessian matrix. As applications, we solve two open problems in dimensions 3 and 4. Finally, we give examples to show that these inequalities cannot be extended to $\mathcal{C}_{k-2}$.
\end{abstract}

\maketitle
\tableofcontents
\section{Introduction}

The classical Sobolev inequality on $\mathbb{S}^n$ states that for $n\ge 3$
\eq{\label{Sobolev_0}
\int_{\mathbb{S}^n}\left(|\nabla u|^2+\frac{n(n-2)}{4}u^2\right)\,dv_{g_{\mathbb{S}^n}}
\ge \frac{n(n-2)}{4}\,\omega_n^{\frac{2}{n}}
\left(\int_{\mathbb{S}^n}u^{\frac{2n}{n-2}}\,dv_{g_{\mathbb{S}^n}}\right)^{\frac{n-2}{n}}.
}
Equivalently, for any conformal metric $g\in[g_{\mathbb{S}^n}]$,
\eq{\label{Soboleb_1}
\frac{\int_{\mathbb{S}^n} R_g\,dv_g}{\bigl(\mathrm{Vol}(g)\bigr)^{\frac{n-2}{n}}}
\ge
\frac{\int_{\mathbb{S}^n} R_{g_{\mathbb{S}^n}}\,dv_{g_{\mathbb{S}^n}}}{\bigl(\mathrm{Vol}(g_{\mathbb{S}^n})\bigr)^{\frac{n-2}{n}}}
= n(n-1)\,\omega_n^{\frac{2}{n}}.
}
Here, $g_{\mathbb{S}^n}$ denotes the round metric, and $\omega_n=\mathrm{Vol}(\mathbb{S}^n,g_{\mathbb{S}^n})$. Geometrically, \eqref{Soboleb_1} is a sharp isoperimetric statement: among conformal metrics with fixed volume, the round metric minimizes the total scalar curvature.

Motivated by this perspective, we study sharp inequalities for the total $\sigma_k$-curvatures. Let $S_g$ be the Schouten tensor,
\[
S_g:=\frac{1}{n-2}\left(\mathrm{Ric}_g-\frac{R_g}{2(n-1)}\,g\right).
\]
Following \cite{ViacDuke}, define
\[
\sigma_k(g):=\sigma_k\bigl(g^{-1}S_g\bigr),
\]
where $\sigma_k$ is the $k$th elementary symmetric function of the eigenvalues $\lambda=(\lambda_1,\dots,\lambda_n)$,
\[
\sigma_k(\lambda)=\sum_{1\le i_1<\cdots<i_k\le n}\lambda_{i_1}\cdots\lambda_{i_k}.
\]
Since $\sigma_1(g)=\frac{1}{2(n-1)}R_g$, the $\sigma_k$-curvatures extend scalar curvature.

We introduce the positivity cones
\[
\mathcal C_k:=\{g:\ \sigma_j(g)>0\ \text{for all }1\le j\le k\},
\qquad
\mathcal C_k[g_0]:=\{g\in[g_0]:\ \sigma_j(g)>0\ \text{for all }1\le j\le k\}.
\]
By convention, $\sigma_0\equiv 1$ and $\mathcal C_0[g_0]=[g_0]$. The total $\sigma_k$-curvature functional is
\eq{\label{eq:Fk-intro}
\Fk(g):=\int_M \sigma_k(g)\,dv_g,
}
where $M$ is the underlying manifold (in this paper, $M=\mathbb{S}^n$ or is a locally conformally flat manifold). In particular, $\mathcal{F}_0(g)=\mathrm{Vol}(g)$. In terms of the $\sigma_k$-notation, the classical Sobolev inequality can be written as
\[
\frac{\int_{\mathbb{S}^n}\sigma_1(g)\,dv_g}{\left(\int_{\mathbb{S}^n}\sigma_0(g)\,dv_g\right)^{\frac{n-2}{n}}}
\ge
\frac{\int_{\mathbb{S}^n} \sigma_1(g_{\mathbb{S}^n})\,dv_{g_{\mathbb{S}^n}}}{\left(\int_{\mathbb{S}^n}\sigma_0(g_{\mathbb{S}^n})\,dv_{g_{\mathbb{S}^n}}\right)^{\frac{n-2}{n}}},
\qquad g\in \mathcal C_0[g_{\mathbb{S}^n}].
\]

A natural problem, generalizing \eqref{Soboleb_1}, is to find sharp inequalities relating $\Fk$ and $\Fl$ for $n\ge k>l\ge 0$. In the cone $\mathcal C_k$ this was largely resolved by Guan--Wang~\cite{GWDuke}; see also \cite{GWCrelle,GVW,GLWJDG,CYCPAM,BVCVPDE} and the discussion at the end of the introduction.

The main goal of this paper is to prove that these sharp inequalities remain valid in the strictly larger cone $\mathcal C_{k-1}$, thereby substantially enlarging the admissible class of conformal metrics.

\begin{thm}
\label{thm:generalized Sobolev inequality-high dim}
Let $0\le l<k\le n$ and let $g\in \mathcal{C}_{k-1}[g_{\mathbb{S}^n}]$ on $(\mathbb{S}^{n},g_{\mathbb{S}^{n}})$. Then:
\begin{enumerate}
\item[\textup{(1)}] \textup{(Sobolev inequality)} If $2k<n$, then
\[
\frac{\int_{\mathbb{S}^n} \sigma_k(g)\,dv_{g}}{\left(\int_{\mathbb{S}^n} \sigma_{l}(g)\,dv_{g}\right)^{\frac{n-2k}{n-2l}}}
\ge
\frac{\int_{\mathbb{S}^n} \sigma_k(g_{\mathbb{S}^n})\,dv_{g_{\mathbb{S}^n}}}{\left(\int_{\mathbb{S}^n} \sigma_{l}(g_{\mathbb{S}^n})\,dv_{g_{\mathbb{S}^n}}\right)^{\frac{n-2k}{n-2l}}}.
\]

\item[\textup{(2)}] \textup{(Moser--Trudinger inequality)} If $2k=n$, then
\[
\mathscr{E}_{n/2}(g)
\ge
\frac{1}{n-2l}\, C_{MT}\left(
\log \int_{\mathbb{S}^n} \sigma_{l}(g)\,dv_g-\log \int_{\mathbb{S}^n}\sigma_{l}(g_{\mathbb{S}^n})\,dv_{g_{\mathbb{S}^n}}
\right),
\]
where $C_{MT}=\frac{\omega_n}{2^{n/2}}\binom{n}{n/2}$ is the sharp constant (achieved by $g=g_{\mathbb{S}^n}$) and
\[
\mathscr{E}_{n/2}(g):=-\int_0^1\int_{\mathbb{S}^n}\sigma_{n/2}(g_t)\,u\,dv_{g_t}\,dt,
\qquad g_t:=e^{-2tu}g_{\mathbb{S}^n}.
\]

\item[\textup{(3)}] \textup{(Reverse inequalities)} If $2k>n$, then:
\begin{enumerate}
\item[\textup{(3.1)}] if $2l\neq n$, the reverse Sobolev inequality holds,
\[
\frac{\int_{\mathbb{S}^n} \sigma_k(g)\,dv_{g}}{\left(\int_{\mathbb{S}^n} \sigma_{l}(g)\,dv_{g}\right)^{\frac{n-2k}{n-2l}}}
\le
\frac{\int_{\mathbb{S}^n} \sigma_k(g_{\mathbb{S}^n})\,dv_{g_{\mathbb{S}^n}}}{\left(\int_{\mathbb{S}^n} \sigma_{l}(g_{\mathbb{S}^n})\,dv_{g_{\mathbb{S}^n}}\right)^{\frac{n-2k}{n-2l}}};
\]
\item[\textup{(3.2)}] if $2l=n$, the reverse Moser--Trudinger inequality holds {  for metric $g$ with  $\int_{\mathbb{S}^n} \sigma_k(g) dv_g>0$}:
\[
\mathscr{E}_{n/2}(g)
\le
-\frac{1}{2k-n}\, C_{MT}\left(
\log \int_{\mathbb{S}^n} \sigma_{k}(g)\,dv_g-\log \int_{\mathbb{S}^n}\sigma_{k}(g_{\mathbb{S}^n})\,dv_{g_{\mathbb{S}^n}}
\right),
\]
or, equivalently, $$e^{-(2k-n)\mathscr{E}_{n/2}(g)/C_{MT}}\ge \frac{\int_{\mathbb{S}^n} \sigma_k(g)dv_g}{\int_{\mathbb{S}^n}\sigma_k(g_{\mathbb{S}^n})\,dv_{g_{\mathbb{S}^n}}}.$$
\end{enumerate}
\end{enumerate}
Moreover, equality holds in each case if and only if $g$ has constant sectional curvature.
\end{thm}

For $k=1$, $l=0$, and $n\ge 3$, Theorem~\ref{thm:generalized Sobolev inequality-high dim} reduces to the classical Sobolev inequality \eqref{Sobolev_0}. For $k=2$ and $l=0$, it was proved in \cite{GWCAG} for $n\ge 5$. Thus, in the sequel we focus on $k\ge 2$ and $l\ge 1$.

In low dimensions, Theorem~\ref{thm:generalized Sobolev inequality-high dim} resolves the following open questions. When $k=2$ and $n=4$, it answers a question raised in \cite{CYCPAM}.

\begin{corollary}\label{k=2andn=4}
On $(\mathbb{S}^{4},g_{\mathbb{S}^{4}})$, let $g=e^{-2u}g_{\mathbb{S}^{4}}\in \mathcal C_{1}$ with $u\in C^{\infty}(\mathbb{S}^4)$. Then
\eq{ \label{ineq_4D}
-\int_{0}^{1}\int_{\mathbb{S}^4}\sigma_{2}(g_{t})\,u\,dv_{g_{t}}\,dt
\ge \frac{1}{4} C_{MT}\log \frac{\mathrm{Vol}(g)}{\mathrm{Vol}(g_{\mathbb S^4})},
}
and
\eq{ \label{ineq_4Db}
-\int_{0}^{1}\int_{\mathbb{S}^4}\sigma_{2}(g_{t})\,u\,dv_{g_{t}}\,dt
\ge \frac{1}{2} C_{MT}\log \frac{\int_{\mathbb{S}^4} \sigma_1(g)\,dv_g}{\int_{\mathbb{S}^4} \sigma_1(g_{\mathbb{S}^4})\,dv_{g_{\mathbb{S}^4}}},
}
where $g_{t}=e^{-2tu}g_{\mathbb{S}^4}$ and $C_{MT}=\frac{3}{2}\omega_4$.
\end{corollary}

When $k=2$ and $n=3$, Theorem~\ref{thm:generalized Sobolev inequality-high dim} answers an open question in \cite{GWAdv} on the standard sphere.

\begin{corollary}\label{k=2andn=3}
On $(\mathbb{S}^{3},g_{\mathbb{S}^{3}})$, let $g\in \mathcal C_{1}[g_{\mathbb{S}^3}]$. Then
\beq \label{eq_a1}
\bigl(\mathrm{Vol}(g)\bigr)^{\frac 1 3}\int_{\mathbb{S}^3} \sigma_2(g)\, dv_g
\le
\bigl(\mathrm{Vol}(g_{\mathbb{S}^{3}})\bigr)^{\frac 1 3}\int_{\mathbb{S}^3} \sigma_2(g_{\mathbb{S}^{3}})\, dv_{g_{\mathbb{S}^{3}}}
=\frac 3 4\,\omega_3^{\frac 43},
\eeq
and
\beq \label{eq_a2}
\left(\int_{\mathbb{S}^3} \sigma_1(g)\, dv_g\right)\left(\int_{\mathbb{S}^3} \sigma_2(g)\, dv_g\right)
\le
\left(\int_{\mathbb{S}^3} \sigma_1(g_{\mathbb{S}^{3}})\, dv_{g_{\mathbb{S}^{3}}}\right)
\left(\int_{\mathbb{S}^3} \sigma_2(g_{\mathbb{S}^{3}})\, dv_{g_{\mathbb{S}^{3}}}\right)
= \frac 9 8\,\omega_3 ^2.
\eeq
\end{corollary}

By the classical Sobolev inequality \eqref{Soboleb_1}, \eqref{ineq_4Db} implies \eqref{ineq_4D}, and \eqref{eq_a2} implies \eqref{eq_a1}.

\begin{rem}\label{rem:k2comments}
\begin{enumerate}
\item[\textup{(a)}] For $k=2$ and $n=4$, inequality \eqref{ineq_4D} answers a question raised by Chang--Yang~\cite{CYCPAM}; see the remark following Theorem~4.3 in \cite{CYCPAM}.
\item[\textup{(b)}] For $k=2$ and $n=3$, the inequalities in Corollary~\ref{k=2andn=3} hold not only on $\mathbb{S}^3$ but also on locally conformally flat $3$-manifolds, answering Conjectures 3 and 4 in \cite{GWAdv} in that setting (related to a question of Viaclovsky). It remains open whether the conformal flatness assumption can be removed.
\end{enumerate}
\end{rem}

\begin{rem}
It is natural to ask whether the inequalities above can be further extended to the larger cone $\mathcal{C}_{k-2}$. In general, they cannot.

For $k=2$, the admissibility condition $g\in\mathcal{C}_1$ is essential. When $n=4$, Chang--Yang~\cite{CYCPAM} constructed examples (large multiples of first eigenfunctions on $\mathbb{S}^4$) for which the inequality fails. When $n\ge 5$, one may take $g=e^{-2\ell x_{n+1}^2}g_{\mathbb S^n}$ (where $x_{n+1}$ is the last coordinate function), so that $\int \sigma_2(g)\,dv_g<0$ for large $\ell$. When $n=3$, the same choice yields
$\bigl(\mathrm{Vol}(g)\bigr)^{\frac 1 3}\int \sigma_2(g)\,dv_g\to +\infty$ as $\ell\to\infty$.
See the Appendix for computations. For $k\ge 3$, one expects analogous failures as well.
\end{rem}

To  state our results more precisely, assume $l\neq\frac n2$ and $k\neq\frac n2$, and recall the following Yamabe-type constants:
\[
\bar Y_{k,l}(M,[g_0]) :=
\begin{cases}
\displaystyle \inf\limits_{g\in \mathcal C_{k}[g_0]} \frac{\int_M \sigma_k(g)\, dv_g}{\left(\int_M \sigma_{l}(g)\, dv_g\right)^{\frac{n-2 k}{n-2l}}}, & \text{if } n>2k,\\
\displaystyle \sup\limits_{g\in \mathcal C_{k}[g_0]} \frac{\int_M \sigma_k(g)\, dv_g}{\left(\int_M \sigma_{l}(g)\, dv_g\right)^{\frac{n-2 k}{n-2l}}}, & \text{if } n<2k,
\end{cases}
\]
\[
Y_{k,l}(M,[g_0]) :=
\begin{cases}
\displaystyle \inf\limits_{g\in \mathcal C_{k-1}[g_0]} \frac{\int_M \sigma_k(g)\, dv_g}{\left(\int_M \sigma_{l}(g)\, dv_g\right)^{\frac{n-2 k}{n-2l}}}, & \text{if } n>2k,\\
\displaystyle \sup\limits_{g\in \mathcal C_{k-1}[g_0]} \frac{\int_M \sigma_k(g)\, dv_g}{\left(\int_M \sigma_{l}(g)\, dv_g\right)^{\frac{n-2 k}{n-2l}}}, & \text{if } n<2k.
\end{cases}
\]
It is clear that $Y_{k,l} \le \bar Y_{k,l}$ for $k<\frac n2$, while $Y_{k,l} \ge \bar Y_{k,l}$ for $k>\frac n2$, and strict inequality can occur. The constant $Y_{k,k-1}(M,[g_0])$ was introduced in \cite{GLWJDG}. By definition, if $\mathcal C_{k}[g_0]=\emptyset$, then $\bar Y_{k,k-1}(M,[g_0])=+\infty$ for $k<\frac{n}{2}$ and $\bar Y_{k,k-1}(M,[g_0])=-\infty$ for $k>\frac{n}{2}$.

One of our main theorems is the following.
\begin{thm}\label{thm:minimizer in weak cone}
Let $n\ge 3$ and $k\ge 2$, and let $(M^{n},g_{0})$ be a  closed locally conformally flat manifold with $g_0\in \mathcal C_{k-1}$. Assume that $k\notin\left\{\frac n2,\frac n2+1\right\}$. If $Y_{k,k-1}(M,[g_0])>0$, then it is achieved by a conformal metric $g\in \mathcal{C}_k[g_0]$, and
\[
Y_{k,k-1}(M,[g_0])=\bar Y_{k,k-1}(M,[g_0])>0.
\]
Moreover, $Y_{k,k-1}(M,[g_0])>0$ if and only if $\mathcal C_{k}[g_0]\neq\emptyset$.
\end{thm}

When $n\ge 5$, the case $k=2$ was proved by Ge--Wang~\cite{GWCAG} using a perturbed heat flow and a technique developed in \cite{GLWJDG}. Those methods do not apply to the low-dimensional cases $n=3,4$. In fact, Theorem~1 in \cite{GWCAG} requires (when $n=3$) an additional hypothesis that $Y_{2,1}(M,[g_0])$ is bounded from above; this boundedness is automatic in dimensions $n\ge 5$. This is one of the main difficulties in dimension~$3$.

It remains open in general whether, for $n=3$, the quantities $Y_{2,1}(M,[g])$ and $Y_{2,0}(M,[g])$ are bounded from above. The analogous questions for $\bar Y_{2,1}$ and $\bar Y_{2,0}$ are also open. Theorem~\ref{thm:minimizer in weak cone} gives an affirmative answer for locally conformally flat manifolds. Recall from Remark~1.5 that there exists a sequence of conformal metrics on $\mathbb{S}^3$ with $({\rm Vol}(g_\ell))^{\frac 13}\int\sigma_2(g_\ell)\,dv_{g_\ell}\to +\infty$.

We also address the cases $Y_{k,k-1}(M,[g_0])<0$ and $Y_{k,k-1}(M,[g_0])=0$ in Section~\ref{Section: Appendix}.

\begin{rem}\label{Yamabeconstantforsphere}
When $M$ is the round sphere, $\mathcal C_k[g_{\mathbb{S}^n}]$ is non-empty for every $k$. Theorem~\ref{thm:minimizer in weak cone} therefore yields
\eq{\label{eq_a0}
Y_{k,k-1}(\mathbb{S}^n,[g_{\mathbb{S}^n}])=\bar Y_{k,k-1}(\mathbb{S}^n,[g_{\mathbb{S}^n}])>0.
}
Hence, when $k\notin \left\{\frac n2,\frac n2+1\right\}$, Theorem~\ref{thm:generalized Sobolev inequality-high dim} follows from Theorem~\ref{thm:minimizer in weak cone} together with the sharp inequalities of Guan--Wang~\cite{GWDuke}, for $l=k-1$. For general $l$ one uses an iteration to prove it.
\end{rem}

The main novelty of this paper is that the geometric inequalities involving $\s_k$ hold in the larger cone $\mathcal{C}_{k-1}$. This enlargement makes the inequalities applicable in a wider range of conformal classes and is   important  in applications. See \cite{GWW1}, \cite{frank2024sharp}, and \cite{Konig-P}. 
Our approach to prove these inequalities is to use fully nonlinear conformal flows. In establishing required  a priori estimates, we encounter several new difficulties, including issues that arose in \cite{GWAdv} but were not resolved there.

As in \cite{GWDuke}, using conformally invariant flow, there is  a crucial advantage, especially on conformally flat manifolds.  One can obtain a uniform gradient bound
\eq{\label{eq:grad-bound-intro}
|\nabla u|\le C,
}
following Ye~\cite{YeJDG}. Consequently, the main analytic task is to establish global second-derivative estimates and to verify that the flow preserves both admissibility and parabolicity. Related ideas were used in \cite{GWCrelle} for a conformal logarithmic $\sigma_k$-Yamabe flow and in \cite{GWDuke} for a conformal logarithmic quotient flow.

We begin by identifying a conformal flow suitable for metrics in the larger cone $\mathcal C_{k-1}$. Recall that in \cite{GWCrelle,GWDuke}, Guan--Wang considered the ``logarithmic'' conformal flows
\begin{equation}\label{flow_log}
-\frac{1}{2}g^{-1}\frac{\partial g}{\partial t}=\log {\sigma_k(g)}-\log r_{k}(g),
\qquad
-\frac{1}{2}g^{-1}\frac{\partial g}{\partial t}=\log \frac{\sigma_k(g)}{\sigma_{k-1}(g)}-\log r_{k,k-1}(g).
\end{equation}
Another natural candidate is
\eq{\label{flow_sigmsk}
-\frac 12 g^{-1}\frac{\partial g}{\partial t}=\sigma_k(g)^{1/k},
\qquad
-\frac 12 g^{-1}\frac{\partial g}{\partial t}=\left(\frac{\sigma_k(g)}{\sigma_{l}(g)}\right)^{\frac 1{k-l}}.
}
However, for conformal metrics lying only in $\mathcal C_{k-1}$, neither of these flows is known to be suitable (for example, admissibility or parabolicity may fail).

To access $\mathcal C_{k-1}$ we remove the logarithmic nonlinearity and introduce, for $\epsilon>0$,
\eq{\label{our flow}
 u_t=-\frac{1}{2}g^{-1}\frac{\partial g}{\partial t}=\frac{\sigma_k(g)-\epsilon}{\sigma_{k-1}(g)}-s,
}
or, equivalently,
\begin{align}
 u_{t} & =e^{2u}\frac{\sigma_{k}(W)-\epsilon e^{-2ku}}{\sigma_{k-1}(W)}-s.
\label{eq:parabolic equation}\nonumber
\end{align}
Here
\eq{\label{ew_W} W=\nabla^{2}u+\nabla u\otimes \nabla u-\frac{|\nabla u|^{2}}{2}g_{0}+S_{g_{0}}.}
The term $s$ is a non-local curvature quantity chosen so that the flow preserves $\int\sigma_{k-1}(g)\,dv_g$.
A crucial observation in \cite{GLWJDG} (see also \cite{GuanZhang,HuiskenSine}) implies that \eqref{our flow} is parabolic in $\mathcal C_{k-1}$ for $\epsilon>0$ (at least for short-time existence); see Section~\ref{subsection:modified  conformal flow} below.
The desired second derivative estimate for \eqref{our flow} presents difficulties absent in the logarithmic setting. In particular, a term of the form $u_iF_i$ appears in the maximum-principle argument (for the definition of $F$ see \eqref{expression of F}); it arises as a mixed term when differentiating \eqref{our flow} twice. In contrast, for the logarithmic flow \eqref{flow_log} this term does not occur.
This is one reason why logarithmic flow \eqref{flow_log} was used for the standard problem $\sigma_k$ -Yamabe in \cite{GWCrelle}, rather than \eqref{flow_sigmsk}. We also remark that for elliptic conformal equations this term does not appear in second-derivative estimates.

To obtain the $C^2$ estimate, we apply the maximum principle to the test function
\[
G=\lambda_{\max}(W)+\left(m+\frac12\right)|\nabla u|^2,
\]
where $\lambda_{\max}$ denotes the maximal eigenvalue of $W$ in \eqref{ew_W}. The second derivative of $\lambda_{\max}$ contains an additional favorable term compared to the second derivative of $W_{11}$; see \eqref{eq_x2}. This extra term is used to control the main unfavorable contribution from the term $u_iF_i$; see \eqref{eq:third derivatives bad term}. Similar ideas have appeared in recent work; see, for instance, \cite{TosattiWeinkoveAJM2010,STWActa2017}.

A further subtlety is that we need estimates for \eqref{our flow} that are uniform in $\epsilon>0$, and hence we must carefully track the terms dependent on $\epsilon$. We manage to show that the sum of these terms is non-negative and thus can be discarded.

To the best of our knowledge, the quotient flow
\eq{\label{our flow_2}
 u_t=-\frac{1}{2}g^{-1}\frac{\partial g}{\partial t}=\frac{\sigma_k(g)}{\sigma_{k-1}(g)},
}
and the $\sigma_k$-Yamabe flow
\eq{\label{our flow_3}
 u_t=\sigma_k(g)^{1/k},
}
have not been previously studied, even in $\mathcal C_{k}$, apart from the logarithmic flow \eqref{flow_log}. As noted above, the logarithmic structure simplifies the $C^2$ estimates. The $C^2$ estimates established here apply equally to \eqref{our flow_2} and \eqref{our flow_3}. 

To show the geometric inequalities on $\S^n$ in Theorem \ref{thm:generalized Sobolev inequality-high dim}, the flow \eqref{our flow} is enough, as mentioned in Remark \ref{Yamabeconstantforsphere}, see also Remark \ref{iteration} below. But for general locally conformally flat manifolds, the iteration method only give an (possible) non-sharp constant. See Remark \ref{iteration}.  To obtain a sharp result in this  case we may consider the  following flow \eq{\label{flow for k-l}
 -\frac{1}{2}g^{-1}\frac{\partial g}{\partial t}=\frac{\sigma_k(g)-r_{k,l}\sigma_l(g)}{\sigma_{k-1}(g)},\quad r_{k,l}=\frac{\int \sigma_k(g) dv_g}{\int \sigma_l(g) dv_g},
}
for $l\le k-1$.
We remark that when $l=k-1$, \eqref{flow for k-l} is not strictly parabolic. This is why we need to consider \eqref{our flow}  for  $l=k-1$.
With same methods in this paper, we can prove  that on a locally conformally flat manifold $(M,g_0)$, this flow converges globally to a limit, provided that $\mathcal C_{k-1}[g_0]\neq \emptyset$ and   $Y_{k,l}(M, [g_0])>0$ for $k,l\neq \frac n2$ and $0\le l<k-1$. As a consequence, we have 
\[ Y_{k,l}(M,[g_0] )=\bar{Y}_{k,l}(M,[g_0]).\]
Together with Theorem \ref{thm:minimizer in weak cone}, we have for  $k,l\neq \frac n2$ and $0\le l<k$,
\[ Y_{k,l}(M,[g_0] )=\bar{Y}_{k,l}(M,[g_0]).\]
This will be useful to determine the exact value of the Yamabe constant $Y_{k,l} $ for tori.

Before concluding the introduction, we briefly recall some related results.

\medskip

\noindent{\it Previous results on geometric inequalities for $k\ge 2$ in $\mathcal{C}_k$.}
\begin{itemize}
\item $k<\frac n2$ (\emph{Sobolev-type}). For $k=1$ this is the classical sharp Sobolev inequality recalled above. Assume now $k\ge 2$. When $l=0$, the inequality was proved in \cite{GVW} using the logarithmic $\sigma_k$ flow introduced in \cite{GWCrelle} together with compactness results for conformally flat manifolds \cite{LLCPAM2003}. For general $l<k\le \frac n2$, it was proved in \cite{GWDuke} via a logarithmic quotient flow.

\item $k=\frac n2$ (\emph{Moser--Trudinger}). When $k=1$ we have $n=2$ and $l=0$, corresponding to the classical Moser--Trudinger--Onofri inequality of Trudinger, Moser and Onofri \cite{TrudingerJMM1967,MoserIU1971,OnofriCMP1982}; see also Hong \cite{HongPAMS1986}. When $n=2k>2$ and $l=0$, Brendle--Viaclovsky and Chang--Yang introduced the critical functional $\mathcal E_{n/2}$ and proved Moser--Trudinger-type inequalities \cite{BVCVPDE,CYCPAM}. Guan--Wang obtained in \cite{GWDuke} the sharp inequality for any $l< k=\frac n2$.

\item $k>\frac n2$ (\emph{Reverse inequalities / quermassintegrals}). Guan--Wang proved in \cite{GWDuke} reverse inequalities for $\frac n2<k\le n$ and $1\le l<k$, and also stated a weaker ``quermassintegral'' inequality. A closely related inequality for $n=4$ and $k=2$ was obtained earlier by Gursky~\cite{GurskyCMP1999}. See also related work of Gursky--Viaclovsky \cite{GVJDG} and \cite{GVW}.
\end{itemize}

\

\noindent\emph{Previous results on related problems and geometric inequalities for $k\ge 2$ and in $\mathcal{C}_{k-1}$}

\medskip
\begin{itemize}
\item In the case $n=4$ and $k=2$, Chang--Gursky--Yang~\cite{CGY2002} proved that $\mathcal{C}_2\neq\emptyset$ if the Yamabe constant $Y_1$ and
$\int_M \sigma_2\bigl(g^{-1}S_g\bigr)\,dv_g$ (which is a conformal invariant when $n=4$) are positive; see also \cite{GVJDG}. When $n=3$ and $k=2$, a similar result was proved in \cite{GLWJDG} and \cite{CatinoDjali2010}. For $n\ge 5$ and $k\ge 2$, Guan--Lin--Wang proved a similar result in \cite{GLWMathZ} for locally conformally flat manifolds.

\item $k<\frac n2$ (\emph{Sobolev-type}). For $k=2$ and $l=1$, Ge--Wang proved Sobolev-type inequalities. When $l=0$, the case $k\ge 1$ was recently proved in \cite{GWW2}, as a consequence of a question raised by J.~Case.
\end{itemize}



\medskip
\noindent{\it Organization of the paper.}
Section~\ref{section:preliminary} reviews notation and basic properties of the $\sigma_k$ functions. We then introduce the conformal flow \eqref{eq:pertubed flow} and prove its basic properties. Section~\ref{estimates forsecond derivatives} establishes the second-derivative estimates for $k\ge 2$. The convergence of the flows and the proofs of the main theorems are given in Section~\ref{Section:Proof of Main theorems}. Section~\ref{Section: Appendix} sketches the proof that non-emptiness of $\mathcal C_k[g_0]$ implies $Y_{k,k-1}(M,[g_0])>0$ for $k<\frac n2$. In the Appendix we also compute examples showing that the main results generally fail in $\mathcal C_{k-2}$.

\section{Preliminaries}\label{section:preliminary}

\subsection{Elementary symmetric functions and G\aa rding cones}

Recall that
\[
\mathcal C_k[g_0]=\bigl\{g\in[g_0]\,\big|\,\lambda(g^{-1}S_g)\in\Gamma_k^+\bigr\},
\]
where $\lambda(g^{-1}S_g)$ denotes the eigenvalues of $g^{-1}S_g$ and
\begin{equation}
\Gamma_k^+=\bigl\{\Lambda=(\lambda_1,\ldots,\lambda_n)\in\mathbb R^n\,\big|\,\sigma_j(\Lambda)>0\ \text{for all }1\le j\le k\bigr\}
\end{equation}
are the G\aa rding cones.

We record several standard identities and inequalities for the elementary symmetric functions $\sigma_k$; see \cite{HuiskenSine}.

\begin{lem}\label{lem:algebra}
Let $\lambda=(\lambda_{1},\ldots,\lambda_{n})\in\mathbb{R}^{n}$ and write
$\lambda\mid i=(\lambda_1,\ldots,\lambda_{i-1},\lambda_{i+1},\ldots,\lambda_n)$.
Then
\begin{align}
\sum_{i=1}^{n}\lambda_{i}\sigma_{k}(\lambda\mid i)&=(k+1)\sigma_{k+1}(\lambda),\label{sigmaproperty1}\\
\sum_{i=1}^{n}\sigma_{k}(\lambda\mid i)&=(n-k)\sigma_{k}(\lambda),\label{sigmaproperty2}\\
\sum_{i=1}^{n}\lambda_{i}^{2}\sigma_{k}(\lambda\mid i)&=\sigma_{1}(\lambda)\sigma_{k+1}(\lambda)-(k+2)\sigma_{k+2}(\lambda),\label{eq:alge 1}\\
(n-k+1)(k+1)\sigma_{k-1}(\lambda)\sigma_{k+1}(\lambda)&\le k(n-k)\sigma_{k}^{2}(\lambda).
\label{eq:alg 2}
\end{align}
If $\lambda\in \Gamma_{k-1}^{+}$, then the Maclaurin inequalities
\[
\bigg(\frac{\sigma_{1}(\lambda)}{n}\bigg)^{k}\ge\frac{\sigma_{k}(\lambda)}{C_{n}^{k}},
\qquad
\left(\frac{\sigma_{k-1}(\lambda)}{C_{n}^{k-1}}\right)^{\frac{k}{k-1}}\ge\frac{\sigma_{k}(\lambda)}{C_{n}^{k}},
\]
hold, and moreover
\begin{equation}
\sigma_{1}(\lambda)\sigma_{k-1}(\lambda)\ge\frac{nk}{n-k+1}\sigma_{k}(\lambda).
\label{eq:alg5}
\end{equation}
\end{lem}

The next lemma is one of the key structural facts used throughout the paper.

\begin{lem}\label{lem:concavity quotient}
Let $1\le k\le n$ and $\lambda\in\Gamma_{k-1}^{+}$. Then the Hessian quotient
\[
F(\lambda)=\frac{\sigma_k(\lambda)}{\sigma_{k-1}(\lambda)}
\]
is (degenerately) elliptic and concave on $\Gamma_{k-1}^{+}$. Equivalently, for any symmetric matrix $W$ with $\lambda(W)\in\Gamma_{k-1}^{+}$, the map
$$
W\longmapsto \frac{\sigma_k(W)}{\sigma_{k-1}(W)}
$$
is concave.
\end{lem}
\begin{proof}
See \cite[Theorem~2.5]{HuiskenSine}, \cite[Lemma~1]{GLWJDG}, or \cite{GuanZhang}.
\end{proof}

\begin{lem}\label{lem:Garding}
For any $\mu,\lambda\in\Gamma_{k-1}^{+}$, we have
\begin{equation}
\frac{\sigma_{k}}{\sigma_{k-1}}(\mu)\le\sum_{i}\mu_{i}\frac{\partial}{\partial\lambda_{i}}\bigg(\frac{\sigma_{k}}{\sigma_{k-1}}(\lambda)\bigg).
\label{eq:garding type}
\end{equation}
Moreover,
\begin{equation}
\sum_{i}\frac{\partial\sigma_{k-1}(\lambda)}{\partial\lambda_{i}}\mu_{i}
\ge (k-1)\frac{\sigma_{k-1}(\mu)^{\frac{1}{k-1}}}{\sigma_{k-1}(\lambda)^{\frac{1}{k-1}-1}}.
\label{eq:garding inequality}
\end{equation}
\end{lem}

\begin{proof}
Both inequalities follow from concavity.

By concavity of $\sigma_{k}/\sigma_{k-1}$ on $\Gamma_{k-1}^{+}$ and \eqref{sigmaproperty2},
\begin{align*}
\frac{\sigma_{k}}{\sigma_{k-1}}(\mu)
&\le  \frac{\sigma_{k}}{\sigma_{k-1}}(\lambda)+\sum_{i}(\mu_{i}-\lambda_{i})\frac{\partial}{\partial\lambda_{i}}\bigg(\frac{\sigma_{k}}{\sigma_{k-1}}(\lambda)\bigg)
=  \sum_{i}\mu_{i}\frac{\partial}{\partial\lambda_{i}}\bigg(\frac{\sigma_{k}}{\sigma_{k-1}}(\lambda)\bigg).
\end{align*}

Again, by concavity of $(\sigma_{k-1})^{\frac{1}{k-1}}$,
\[
(\sigma_{k-1})^{\frac{1}{k-1}}(\mu)
\le(\sigma_{k-1})^{\frac{1}{k-1}}(\lambda)+\sum_{i}(\mu_{i}-\lambda_{i})\frac{\partial}{\partial\lambda_{i}}(\sigma_{k-1})^{\frac{1}{k-1}}(\lambda),
\]
which implies \eqref{eq:garding inequality}.
\end{proof}





\subsection{A perturbed conformal flow}\label{subsection:modified  conformal flow}

Ideally, one would like to study the conformal flow
\begin{equation}
 u_t=-\frac{1}{2}g^{-1}\frac{\partial g}{\partial t}
 =\frac{\sigma_{k}(g)}{\sigma_{k-1}(g)}-s
\quad\text{on }M,
\end{equation}
with a normalization term $s$. Throughout, we write $g=e^{-2u}g_0$. However, Lemma~\ref{lem:concavity quotient} only yields weak parabolicity for this flow. Therefore, we consider the following perturbed flow
\begin{equation}
 u_t=-\frac{1}{2}g^{-1}\frac{\partial g}{\partial t}
 =\frac{\sigma_{k}(g)-\epsilon}{\sigma_{k-1}(g)}-s_{\epsilon}
\quad\text{on }M,
\label{eq:pertubed flow}
\end{equation}
with initial data $g(0)=g_0$, where
\[
 s_{\epsilon}=\frac{\int_M (\sigma_{k}(g)-\epsilon)\,dv_{g}}{\int_M \sigma_{k-1}(g)\,dv_{g}}.
\]
By construction, the flow preserves the quantity $\int_M \sigma_{k-1}(g)\,dv_g$ (see Lemma~\ref{lem:fucntional increase and r constant pertubed flow} below). For every $\epsilon>0$, the flow is parabolic and remains conformally invariant.

We also modify the associated functionals. Define
\[
J^k_{\epsilon}(g)=\int_M \sigma_{k}(g)\,dv_{g}+\frac{(2k-n)\epsilon}{n}\,\Vol(g)
\]
and
\[
\mathscr{E}_{n/2,\epsilon}(g)=-\int_{0}^{1}\int_{M}\sigma_{n/2}(g_{s})\,u\,dv_{g_{s}}\,ds-\frac{\epsilon}{n}\Vol(g),
\]
where $g_s=e^{-2su}g_0$ connects $g_0$ to $g$.

\begin{lem}\label{lem:fucntional increase and r constant pertubed flow}
Assume that $g\in\mathcal C_{k-1}$. If $k<\frac{n}{2}$, then $J^k_{\epsilon}$ is nonincreasing along the flow~\eqref{eq:pertubed flow}. If $k>\frac{n}{2}$, then $J^k_{\epsilon}$ is nondecreasing along the flow~\eqref{eq:pertubed flow}. If $k=\frac{n}{2}$, then $\mathscr{E}_{n/2,\epsilon}(g)$ is nonincreasing along the flow~\eqref{eq:pertubed flow}.
In particular,
\begin{equation}\label{monotoneofJe}
\frac{\partial}{\partial t} J^k_{\epsilon}
= (2k-n)\int_M\Bigl(\frac{\sigma_{k}(g)-\epsilon}{\sigma_{k-1}(g)}-s_{\epsilon}\Bigr)^{2}\sigma_{k-1}(g)\,dv_{g},
\end{equation}
\begin{equation}
\frac{\partial} {\partial t}\mathscr{E}_{n/2,\epsilon}(g)
=-\int_M\Bigl(\frac{\sigma_{\frac n 2}(g)-\epsilon}{\sigma_{\frac n 2-1}(g)}-s_{\epsilon}\Bigr)^2\sigma_{\frac{n}{2}-1}(g)\,dv_{g},
\end{equation}
and
\[
\frac {\partial }{\partial t}\int_M\sigma_{k-1}(g)\,dv_{g}= 0.
\]
\end{lem}

\begin{proof}
By the choice of $s_{\epsilon}$ and a variational formula proved Viaclovsky \cite{ViacDuke}, we have
\begin{align*}
\frac{d}{dt}\int_M\sigma_{k-1}(g)\,dv_{g}
& =(2(k-1)-n)\int_M\sigma_{k-1}(g)u_{t}\,dv_g\\
& =(2(k-1)-n)\int_M\sigma_{k-1}(g)\Bigl(\frac{\sigma_{k}(g)-\epsilon}{\sigma_{k-1}(g)}-s_{\epsilon}\Bigr)\,dv_g\\
& =(2(k-1)-n)\int_M\bigl(\sigma_{k}(g)-\epsilon-s_{\epsilon}\sigma_{k-1}(g)\bigr)\,dv_{g}=0.
\end{align*}
If $2(k-1)=n$, then we  deduce from \cite{BVCVPDE}
\[
\frac{d}{dt}\left(-\int_{0}^{1}\int_{M}\sigma_{n/2}\left(g_{s}\right)u\,dv_{g_{s}}\,ds\right)=0.
\]
Furthermore,
\begin{align*}
\frac{d}{dt}\int_M\sigma_{k}(g)\,dv_{g}
& =(2k-n)\int_M\sigma_{k}(g)u_{t}\,dv_{g}\\
& =(2k-n)\int_M\sigma_{k}(g)\Bigl(\frac{\sigma_{k}(g)-\epsilon}{\sigma_{k-1}(g)}-s_{\epsilon}\Bigr)\,dv_{g},
\end{align*}
and
\begin{align*}
\frac{d}{dt}\Vol(g)
& =-n\int_M u_{t}\,dv_{g}
=-n\int_M\Bigl(\frac{\sigma_{k}(g)-\epsilon}{\sigma_{k-1}(g)}-s_{\epsilon}\Bigr)\,dv_{g}.
\end{align*}
Therefore,
\begin{align*}
\frac{d}{dt}\left(\int_M\sigma_{k}(g)\,dv_{g}+\frac{(2k-n)\epsilon}{n}\Vol(g)\right)
&= (2k-n)\int_M\Bigl(\frac{\sigma_{k}(g)-\epsilon}{\sigma_{k-1}(g)}-s_{\epsilon}\Bigr)^{2}\sigma_{k-1}(g)\,dv_{g}.
\end{align*}
Similarly, when $2k=n$ we have {
\begin{align*}
\frac{d\mathscr{E}_{n/2,\epsilon}(g)}{dt}
&=-\int_M\Bigl(\frac{\sigma_{n/2}(g)-\epsilon}{\sigma_{n/2-1}(g)}-s_{\epsilon}\Bigr)(\sigma_{n/2}(g)-\epsilon)\,dv_{g}.
\end{align*}
}
Using the definition of $s_{\epsilon}$ and the evolution of $\Vol(g)$, we obtain
\begin{align*}
\frac{d\mathscr{E}_{n/2,\epsilon}(g)}{dt}
&=-\int_M\Bigl(\frac{\sigma_{\frac n 2}(g)-\epsilon}{\sigma_{\frac n 2-1}(g)}-s_{\epsilon}\Bigr)^2\sigma_{\frac{n}{2}-1}(g)\,dv_{g}.
\end{align*}
\end{proof}

When $W$ is a symmetric matrix, we write $\lambda(W)$ for the vector of eigenvalues of $W$ and define $\sigma_k(W):=\sigma_k(\lambda(W))$.
The flow~\eqref{eq:pertubed flow} can be written in the form
\begin{align}
 u_{t} &= e^{2u}F(W,u)-s_{\epsilon},
\label{eq:parabolic equation}
\end{align}
where
\eq{\label{expression of F}
F(W,u) := \frac{\sigma_{k}(W)-\nu}{\sigma_{k-1}(W)},
}
with
\[
\nu:=\epsilon e^{-2ku},
\]
and
\[
W:=S_{g_u}=\nabla^{2}u+\nabla u\otimes \nabla u-\frac{|\nabla u|^{2}}{2}g_{0}+S_{g_{0}}.
\]

Denote
\[
F^{ij}:= \frac{\partial}{\partial W_{ij}}F(W, u)
=\frac{1}{\sigma_{k-1}^{2}(W)}\left\{ \sigma_{k-1}(W)T_{k-1}(W)_{j}^{i}-(\sigma_{k}(W)-\nu)T_{k-2}(W)_{j}^{i}\right\}.
\]
Then
\begin{align}
\sum_{i}F^{ii}
&=\left(n-k+1-\frac{(n-k+2)\sigma_{k}(W)\sigma_{k-2}(W)}{\sigma_{k-1}^{2}(W)}\right)+\frac{(n-k+2)\nu\sigma_{k-2}(W)}{\sigma_{k-1}^{2}(W)}\nonumber \\
&=\sum_{i}\frac{\partial}{\partial W_{ii}}\frac{\sigma_{k}(W)}{\sigma_{k-1}(W)}+\frac{(n-k+2)\nu\sigma_{k-2}(W)}{\sigma_{k-1}^{2}(W)}.
\label{eq:F^ii trace}
\end{align}
By Lemma~\ref{lem:concavity quotient}, for each fixed $u$ (hence $\nu\ge 0$) the operator $W\mapsto F(W,u)$ is elliptic and concave on $\Gamma_{k-1}^{+}$.

Define the maximal existence time of~\eqref{eq:pertubed flow} by
\[
T^*=\sup \left\{T_0>0\,\middle|\,\eqref{eq:pertubed flow}\ \text{exists on }[0,T_0]\ \text{and }g(t)\in\mathcal{C}_{k-1}\ \text{for }t\in[0,T_0]\right\}.
\]
The parabolicity of~\eqref{eq:pertubed flow} implies that $T^*>0$.

\begin{lem}\label{lem:gradient est}
On a locally conformally flat manifold $(M,g_{0})$, let $g=e^{-2u}g_{0}\in \mathcal C_{k-1}$ satisfy \eqref{eq:pertubed flow} for $t\in [0, T]$ with $T<T^*$. Then
\[
|\nabla u|\le C,
\]
where $C$ is independent of $\epsilon$, $s_{\epsilon}$, and $T$.
\end{lem}
\begin{proof}
Applying Ye's argument \cite{YeJDG}, we obtain $|\nabla u|\le C$. This is one advantage of using a conformal invariant flow on a locally conformally flat manifold.
\end{proof}

\section{Second derivative estimates}\label{estimates forsecond derivatives}

Our main task is to establish global second-derivative estimates.

\begin{thm}\label{thm:C2 estimates}
Let $g=e^{-2u}g_{0}\in \mathcal C_{k-1}[g_{0}]$ be a solution of \eqref{eq:pertubed flow} on $M\times[0,T]$, and assume that $|\nabla u|\le c$ and $\epsilon>0$. Then
\[
|\nabla^{2}u|\le C,
\]
where $C$ depends only on $n$, $k$, $g_0$, and $\|\nabla u\|_{C^{0}}$, and is independent of $\epsilon$ and $s_{\epsilon}$.
\end{thm}

\begin{proof}
Set
\[
G=\lambda_{\max}(g_0^{-1}W)+\left(m+\frac12\right)|\nabla u|^2,
\qquad m:=4n.
\]
Assume $G$ attains its maximum at $(x_{0},t_{0})\in M\times(0,T]$. Let $\lambda_1$ denote the largest eigenvalue of $g_0^{-1}W$ at $(x_0,t_0)$. Since $\lambda_1$ may have multiplicity, we use a perturbation argument (cf. \cite{szekelyhidi2018fully,STWActa2017}).

In a neighborhood $U$ of $x_0$, choose $g_0$-normal coordinates $(x^1,\ldots,x^n)$ centered at $x_0$ and adapted to an eigenbasis of $W(x_0,t_0)$. Then
\[
(g_0)_{ij}(x_0)=\delta_{ij},
\qquad
\partial_k(g_0)_{ij}(x_0)=0,
\qquad
\Gamma_{ij}^k(x_0)=0.
\]
We may assume $W$ is diagonal at $(x_0,t_0)$ and with no confusions, we denote $W_{1}^1=\lambda_{\max}(W)$ and also $W^{1}_1=W_{11}$ . If $G(x_0,t_0)$ is bounded by a universal constant depending only on $n,k,\|\nabla u\|_{C^0}$ and $\|S_{g_0}\|_{C^2}$, then the desired estimate is immediate; thus we may further assume $G(x_0,t_0)$ is sufficiently large so that the inequalities used below hold.

Define the perturbation tensor $B$ in these coordinates by constant components
\[
B_{ij}=\delta_{ij}-\delta_{i1}\delta_{j1}
\qquad\bigl(\text{i.e. }B_{11}=0,\ B_{1p}=B_{p1}=0,\ B_{pq}=\delta_{pq}\text{ for }p,q>1\bigr).
\]
Then
\[
\partial_k B_{ij}=0,
\qquad
\partial_l\partial_k B_{ij}=0.
\]

Denote $\widetilde W=W-B$.
With the above choice of \(B\), at \(x_0\) we have for all $1\le i,p\le n$
\[
\widetilde W_{11,i}=W_{11,i},
\qquad
\widetilde W_{1p,i}=W_{1p,i},
\qquad
\widetilde W_{11,ii}=W_{11,ii}.
\]
We denote by
\[
\widetilde\lambda_1\ge\widetilde\lambda_2\ge\cdots\ge\widetilde\lambda_n
\]
the eigenvalues of \(\widetilde W\). At \((x_0,t_0)\), we have
\[
\widetilde\lambda_1=\lambda_1,\qquad
\widetilde\lambda_p=\lambda_p-1,\quad p>1.
\]
Therefore
\[
\widetilde\lambda_1-\widetilde\lambda_p
=
\lambda_1-\lambda_p+1\ge 1,\qquad p>1.
\]
Hence \(\widetilde\lambda_1\) is a simple eigenvalue of \(\widetilde W\)
at \((x_0,t_0)\), and consequently \(\widetilde\lambda_1\) is smooth in a
neighborhood of \((x_0,t_0)\).

We define the local perturbed auxiliary function
\[
\widetilde G
=
\widetilde\lambda_1+\left(m+\frac12\right)|\nabla u|^2
\]
on \(U\times [0,T]\). 

Since $B\ge 0$ (as a bilinear form), we have $\widetilde W=W-B\le W$ and hence
\[
\widetilde\lambda_1=\lambda_{\max}(\widetilde W)\le \lambda_{\max}(W)
\quad\text{near }(x_0,t_0).
\]
At $(x_0,t_0)$ we have $\widetilde\lambda_1=\lambda_{\max}(W)$, and therefore
\[
\widetilde G(x_0,t_0)=G(x_0,t_0),
\qquad
\widetilde G\le G\ \text{near }(x_0,t_0).
\]
Since $G$ attains its maximum at $(x_0,t_0)$, $\widetilde G$ attains a local maximum at $(x_0,t_0)$ as well. Hence we may apply the maximum principle to $\widetilde G$ at $(x_0,t_0)$ and obtain
\[
0\le \widetilde G_t,\qquad
0=\widetilde G_i,\qquad
0\ge \sum_{i,j=1}^nF^{ij}\partial_i \partial_j\widetilde  G=\sum_{i,j=1}^nF^{ij}\nabla_i \nabla_j\widetilde  G.
\]
Equivalently, at $(x_0,t_0)$,
\begin{align}
0 &\le \widetilde\lambda_{1,t}+(2m+1)\sum_{k=1}^n u_k u_{kt},
\label{first derivative of time}\\
0 &= \widetilde\lambda_{1,i}+(2m+1)\sum_{\ell=1}^n u_{\ell}u_{\ell i},
\label{first derivative of space}\\
0 &\ge \sum_{i,j=1}^n F^{ij}(\widetilde\lambda_1)_{ij}
+(2m+1)\sum_{i,j,\ell=1}^nF^{ij}u_{\ell i}u_{\ell j}
+(2m+1)\sum_{i,j,\ell=1}^nF^{ij}u_{\ell}u_{\ell ij}.
\label{second derivative of space}
\end{align}

At \((x_0,t_0)\), since \(\widetilde W\) is diagonal and
\(\widetilde\lambda_1\) is simple,  we have
\[
\frac{\partial \widetilde\lambda_1}{\partial \widetilde W^p_{q}}
=
\delta_{p1}\delta_{q1},
\]
and
\[
\frac{\partial^2 \widetilde\lambda_1}
{\partial \widetilde W^p_{q}\partial \widetilde W^r_{s}}
=
(1-\delta_{p1})
\frac{\delta_{q1}\delta_{r1}\delta_{ps}}
{\widetilde\lambda_1-\widetilde\lambda_p}
+
(1-\delta_{r1})
\frac{\delta_{s1}\delta_{p1}\delta_{qr}}
{\widetilde\lambda_1-\widetilde\lambda_r}.
\]
See \cite{spruck2005geometric} and \cite{szekelyhidi2018fully}.
Therefore, at $(x_0,t_0)$, 
{
\[
\widetilde\lambda_{1,i}=W_{11,i},\; \widetilde\lambda_{1,t}=W_{11,t}
\]
}
and
\eq{\label{eq_x2}
(\widetilde\lambda_1)_{ii}
&=\frac{\partial^2 \widetilde\lambda_1}
{\partial \widetilde W^p_{q}\partial \widetilde W^r_{s}} \widetilde W^p_{q,i}\widetilde W^r_{s,i}+\frac{\partial \tilde \lambda_1}{\partial \widetilde W^p_{q}}\widetilde W^p_{p,ii}\\
&=
W^1_{1,ii}
+
2\sum_{p>1}
\frac{W_{1p,i}^2}
{\lambda_1-\widetilde\lambda_p},
}
where $\nabla_i W^p_q=\nabla_i W_{pq}$ at $(x_0,t_0)$.
Consequently,
\[
\sum_{i=1}^n F^{ii}(\widetilde\lambda_1)_{ii}
=
\sum_{i=1}^nF^{ii}W_{11,ii}
+
2\sum_{i=1}^nF^{ii}\sum_{p>1}
\frac{W_{1p,i}^2}
{\lambda_1-\widetilde\lambda_p}.
\]

Now by \eqref{first derivative of time}, \eqref{first derivative of space}, and \eqref{second derivative of space}, we have
\begin{align}
0\ge {}& e^{2u}\bigg(\sum_{i=1}^n F^{ii}\left(\widetilde{\lambda}_{1}\right)_{ii}+(2m+1)\sum_{i,\ell=1}^nF^{ii}u_{\ell i}u_{\ell i}+(2m+1)\sum_{i,\ell=1}^nF^{ii}u_{\ell}u_{\ell ii}\bigg)\nonumber\\
&-\Bigl(W_{11,t}+(2m+1)\sum_{k=1}^n u_{k}u_{kt}\Bigr)\nonumber\\
= {}& e^{2u}\left(\sum_{i=1}^nF^{ii}\Bigl(W_{11,ii}+2\sum_{p>1}\frac{W_{1p,i}^{2}}{\lambda_{1}-\widetilde{\lambda}_{p}}\Bigr)
+(2m+1)\sum_{i,\ell=1}^nF^{ii}u_{\ell i}u_{\ell i}+(2m+1)\sum_{i,\ell=1}^nF^{ii}u_{\ell}u_{\ell ii}\right)\nonumber\\
&-\Bigl(W_{11,t}+(2m+1)\sum_{k=1}^n u_{k}u_{kt}\Bigr)\nonumber\\
\ge {}& e^{2u}\sum_{i=1}^nF^{ii}\Bigl(u_{11ii}+2u_{1}u_{1ii}+2u_{1i}u_{1i}+(S_{g_{0}})_{11,ii}\Bigr)
+e^{2u}\sum_{i,\ell=1}^nF^{ii}\Bigl(2m u_{\ell i}u_{\ell i}+2m u_{\ell}u_{\ell ii}\Bigr)\nonumber\\
&-\Bigl(u_{11t}+2u_{1}u_{1t}+2m\sum_{\ell=1}^n u_{\ell}u_{\ell t}\Bigr)
+2e^{2u}\sum_{i=1}^nF^{ii}\sum_{p>1}\frac{W_{1p,i}^{2}}{\lambda_{1}-\widetilde{\lambda}_{p}}\nonumber\\
\ge {}& e^{2u}\sum_{i=1}^n\Bigl(F^{ii}W_{ii,11}+2F^{ii}u_{1}W_{ii,1}+2m\sum_{\ell=1}^nF^{ii}u_{\ell}W_{ii,\ell}
+\sum_{\ell=1}^nF^{ii}u_{1\ell}^{2}+2m\sum_{\ell=1}^nF^{ii}u_{\ell i}u_{\ell i}\Bigr)\nonumber\\
&-\Bigl(u_{11t}+2u_{1}u_{1t}+2m\sum_{\ell=1}^n u_{\ell}u_{\ell t}\Bigr)
+2e^{2u}\sum_{i=1}^nF^{ii}\sum_{p>1}\frac{W_{1p,i}^{2}}{\lambda_{1}-\widetilde{\lambda}_{p}}-C(1+|\nabla u|^2)e^{2u}G F^{ii}.
\label{mainineq1}
\end{align}
Here we used the standard commutation identities and estimates
\begin{align*}
u_{lij}&=u_{ijl}+\sum_{m=1}^nR_{milj}u_{m},\\
u_{llij}&= u_{ijll}+O(|\nabla^2 u|)+O(|\nabla u|),\\
\left(\sum_{\ell=1}^nu_{\ell}^{2}\right)_{pp}&=2\sum_{\ell=1}^n\left(u_{pp\ell}u_{\ell}+u_{p\ell}^{2}\right)+O\left(|\nabla u|^{2}\right),\\
u_iu_{i11}&=O\bigl((1+|\nabla u|^{2})G\bigr).
\end{align*}
For any fixed $p$,
\begin{align*}
F_{p}u_{p}=\sum_{i,j=1}^{n}F^{ij}W_{ij,p}u_{p}+\frac{\partial F}{\partial u}u_{p}^{2}=\sum_{i,j=1}^{n}F^{ij}W_{ij,p}u_{p}+\frac{2k\epsilon e^{-2ku}u_{p}^{2}}{\sigma_{k-1}(W)},
\end{align*}
and
\[
\begin{aligned}\sum_{i,j=1}^{n}F^{ij}W_{ij,pp}= & F_{pp}-\sum_{i,j,s,l=1}^{n}\frac{\partial^{2}F}{\partial W_{ij}\partial W_{sl}}W_{ij,p}W_{sl,p}\\
 & -2\sum_{i,j=1}^{n}\frac{\partial^{2}F}{\partial W_{ij}\partial u}W_{ij,p}u_{p}-\frac{\partial^{2}F}{\partial^{2}u}u_{p}^{2}-\frac{\partial F}{\partial u}u_{pp}\\
= & F_{pp}-\sum_{i,j,s,l=1}^{n}\frac{\partial^{2}F}{\partial W_{ij}\partial W_{sl}}W_{ij,p}W_{sl,p}-2\sum_{i,j=1}^{n}\frac{\partial^{2}F}{\partial W_{ij}\partial u}W_{ij,p}u_{p}\\
 & +\frac{(2k)^{2}\epsilon e^{-2ku}u_{p}^{2}}{\sigma_{k-1}(W)}-\frac{2k\epsilon e^{-2ku}u_{pp}}{\sigma_{k-1}(W)}.
\end{aligned}
\]

It follows 
\begin{align}
 & e^{2u}\sum_{i=1}^nF^{ii}W_{ii,11}+2e^{2u}\sum_{i=1}^nF^{ii}u_{1}W_{ii,1}+2e^{2u}m\sum_{i,l=1}^{n}F^{ii}u_{l}W_{ii,l}
\nonumber\\
= & e^{2u}(F_{11}+2F_{1}u_{1}+2m\sum_{l}F_{l}u_{l}
)\nonumber\\
 & +e^{2u}\bigg(-\sum_{i,j,s,l=1}^{n}\frac{\partial^{2}F}{\partial W_{ij}\partial W_{sl}}W_{ij,1}W_{sl,1}-2\sum_{i,j=1}^{n}\frac{\partial^{2}F}{\partial W_{ij}\partial u}W_{ij,1}u_{1}\bigg)\nonumber\\
 & -e^{2u}\frac{2k\epsilon e^{-2ku}}{\sigma_{k-1}(W)}(u_{11}+(2-2k)u_{1}^{2}+2m|\nabla u|^{2}).\label{fourthorderderivative}
\end{align}
Differentiating our equation 
\begin{equation}
e^{2u}F(W,u)-s_{\epsilon}=u_{t}\label{eq:equation with time2}
\end{equation}
yields that for any $p=1,\cdots,n$, 
\begin{equation}
2e^{2u}u_{p}F+e^{2u}F_{p}=u_{pt},\label{eq:first derivative of equation2}
\end{equation}
and 
\begin{equation}
4e^{2u}u_{p}^{2}F+2e^{2u}u_{pp}F+4e^{2u}u_{p}F_{p}+e^{2u}F_{pp}=u_{ppt}.\label{eq:second derivative of equation}
\end{equation}
Combining (\ref{eq:first derivative of equation2}) and (\ref{eq:second derivative of equation})
yields 
\begin{align}
 & e^{2u}\big(F_{11}+2F_{1}u_{1}+2m\sum_{l}F_{l}u_{l}
\big)-(u_{11t}+2u_{1}u_{1t})-2m\sum_{l=1}^{n}u_{l}u_{lt}\nonumber\\
= & -(8e^{2u}u_{1}^{2}F+2e^{2u}u_{11}F+4e^{2u}u_{1}F_{1})-4me^{2u}|\nabla u|^{2}F.\label{secondderivativeofequation}
\end{align}
Now by (\ref{mainineq1}), (\ref{fourthorderderivative}),
(\ref{secondderivativeofequation}),
we arrive at 
{
\begin{equation}
\begin{aligned}0\geq & \sum_{i,j=1}^{n}e^{2u}F^{ij}\widetilde{G}_{ij}-\widetilde{G}_{t}\\
\ge & -(8e^{2u}u_{1}^{2}F+2e^{2u}u_{11}F+4e^{2u}u_{1}F_{1})-4me^{2u}|\nabla u|^{2}F\\
 & +e^{2u}\bigg(-\sum_{i,j,s,l=1}^{n}\frac{\partial^{2}F}{\partial W_{ij}\partial W_{sl}}W_{ij,1}W_{sl,1}-2\sum_{i,j=1}^{n}\frac{\partial^{2}F}{\partial W_{ij}\partial u}W_{ij,1}u_{1}\bigg)\\
 & +e^{2u}\sum_{i=1}^{n}F^{ii}\sum_{l=1}^{n}u_{1l}^{2}+2me^{2u}\sum_{i,j,l=1}^{n}F^{ij}u_{li}u_{lj}+2e^{2u}\sum_{i=1}^{n}F^{ii}\sum_{p>1}\frac{W_{1p,i}^{2}}{\lambda_{1}-\widetilde{\lambda}_{p}}\\
 & -e^{2u}\frac{2k\epsilon e^{-2ku}}{\sigma_{k-1}(W)}(u_{11}+(2-2k)u_{1}^{2}+2m|\nabla u|^{2})-Ce^{2u}(1+|\nabla u|^2)G\sum_{i=1}^{n}F^{ii}.
\end{aligned}
\label{eq:final 1}
\end{equation}
}
We remark that the term 
\begin{align}
-4e^{2u}u_{1}F_{1}
\end{align}in the previous inequality 
is the extra one we have mentioned in the Introduction and have to deal
with. It can be written as 
\begin{align}
-4e^{2u}u_{1}F_{1} & =-4e^{2u}u_{1}(\sum_{i,j=1}^{n}F^{ij}W_{ij,1}+F_{u}u_{1})\nonumber \\
 & =-4e^{2u}u_{1}\sum_{i=1}^{n}F^{ii}W_{ii,1}-\frac{8k\epsilon e^{2u}e^{-2ku}u_{1}^{2}}{\sigma_{k-1}(W)}\nonumber \\
 & \ge-4e^{2u}u_{1}\sum_{i=1}^{n}F^{ii}W_{i1,i}-\frac{8k\epsilon e^{2u}e^{-2ku}u_{1}^{2}}{\sigma_{k-1}(W)}-Ce^{2u}\sum_{i=1}^{n}F^{ii}G(1+|\nabla u|^2)\nonumber \\
 & \ge-e^{2u}\sum_{i=2}^{n}\frac{F^{ii}W_{1i,i}^{2}}{\lambda_{1}-\tilde{\lambda}_{i}}-Ce^{2u}u_1^2\sum_{i=2}^{n}F^{ii}(\lambda_{1}-\tilde{\lambda}_{i})+4e^{2u}u_{1}F^{11}((2m+1)\sum_{l=1}^{n}u_{l}u_{l1})\nonumber\\
 &\quad -\frac{8k\epsilon e^{2u}e^{-2ku}u_{1}^{2}}{\sigma_{k-1}(W)}-Ce^{2u}\sum_{i=1}^{n}F^{ii}G(1+|\nabla u|^2)\nonumber\\
&\ge  -e^{2u}\sum_{i=2}^{n}\frac{F^{ii}W_{1i,i}^{2}}{\lambda_{1}-\tilde{\lambda}_{i}}-\frac{8k\epsilon e^{2u}e^{-2ku}u_{1}^{2}}{\sigma_{k-1}(W)}-Ce^{2u}\sum_{i=1}^{n}F^{ii}G(1+|\nabla u|^2),\label{eq:third derivatives bad term}
\end{align}
where in the fourth inequality we used the facts $W_{11,1}+(2m+1)\sum_{i=1}^n u_{1i}u_i=G_1=0$ and {$u_{li}=W_{li}+O(1+|\nabla u|^2)=O(G)$.} 
Recall $\nu=\epsilon e^{-2ku}\ge0$ and $F(W)=\frac{\sigma_k(W)-\nu}{\sigma_{k-1}(W)}$. It holds that
$$\frac{\partial^{2}F}{\partial W_{ij}\partial u}=-\frac{\partial \sigma_{k-1}(W)}{\partial W_{ij}}\frac{2k\nu }{\sigma_{k-1}^2(W)},$$
and hence,
\begin{align}\label{third bad term}
 -2\sum_{i,j=1}^{n}\frac{\partial^{2}F}{\partial W_{ij}\partial u}W_{ij,1}u_{1}\nonumber
= & \frac{4k\nu}{\sigma_{k-1}^{2}}(\sigma_{k-1}(W))_{1}u_{1}\nonumber\\
\ge & \nu\left(-\frac{(\sigma_{k-1}(W))_{1}(\sigma_{k-1}(W))_{1}}{\sigma_{k-1}(W)^{3}}-\frac{4k^{2}u_{1}^{2}}{\sigma_{k-1}(W)}\right).
\end{align}

Since  {$\frac{\sigma_{k}(W)}{\sigma_{k-1}(W)}$ and $(\sigma_{k-1}(W))^{\frac{1}{k-1}}$ are }concave for $W\in\Gamma_{k-1}^{+}$, we have 
\begin{equation}\label{goodthirdterm}
-\sum_{i,j,s,l=1}^{n}\frac{\partial^{2}F}{\partial W_{ij}\partial W_{sl}}W_{ij,1}W_{sl,1}\ge
\frac{k}{k-1}\nu\frac{(\sigma_{k-1}(W))_{1}(\sigma_{k-1}(W))_{1}}{\sigma_{k-1}(W)^{3}}.
\end{equation}
See the proof of \eqref{goodthirdterm} in \cite{GLL,GuanZhang} for example.

Inserting (\ref{eq:third derivatives bad term}), \eqref{third bad term} and  \eqref{goodthirdterm} into (\ref{eq:final 1}), we get 
\begin{equation}
\begin{aligned}0\geq & \sum_{i,j=1}^{n}e^{2u}F^{ij}\widetilde{G}_{ij}-\widetilde{G}_{t}\\
\ge & -(8e^{2u}u_{1}^{2}F+2e^{2u}u_{11}F)-4me^{2u}|\nabla u|^{2}F\\
 & +e^{2u}\sum_{i=1}^{n}F^{ii}\sum_{l=1}^{n}u_{1l}^{2}+2me^{2u}\sum_{i,j,l=1}^{n}F^{ij}W_{li}W_{lj}\\
 & -e^{2u}\frac{2k\nu}{\sigma_{k-1}(W)}u_{11}(1+\frac{1}{8k})-Ce^{2u}G(1+|\nabla u|^2)\sum_{i=1}^{n}F^{ii},
\end{aligned}
\label{eq:final 1-1}
\end{equation}
where we have used that {$u_{11}\ge 104nk|\nabla u|^{2}$}.

By (\ref{eq:alge 1}),(\ref{eq:alg 2}) and (\ref{eq:alg5}), we obtain 
\begin{equation}
\sum_{i,j,l=1}^{n}F^{ij}W_{li}W_{lj}\ge\frac{k}{n-k+1}\frac{\sigma_{k}(W)^{2}}{\sigma_{k-1}(W)^{2}}+\frac{(k-1)}{n}\frac{\nu\sigma_{1}(W)}{\sigma_{k-1}(W)}\label{eq:FiiWii^2-1}
\end{equation}
and also
\begin{equation}\label{k2}
\sum_{i=1}^{n}F^{ii}W_{ii}^{2}\ge\frac{2\sigma_{2}(W)^{2}}{(n-1)\sigma_{1}(W)^{2}}+\frac{\nu\sum_{i=1}^{n}W_{ii}^{2}}{\sigma_{1}(W)^{2}}\quad\text{for}\quad k=2.
\end{equation}
Due to \eqref{eq:F^ii trace} and \eqref{eq:garding type}, it holds that 
\[
\sum_{i=1}^{n}F^{ii}\ge\frac{n-k+1}{k}+(n-k+2)\frac{\nu\sigma_{k-2}}{\sigma_{k-1}^{2}}.
\]
Then, by the Cauchy-Schwarz inequality and \eqref{eq:FiiWii^2-1}, we have 
\begin{align}
 & -3u_{11}\left|\frac{\sigma_{k}(W)}{\sigma_{k-1}(W)}\right|+m\sum_{i,j,l=1}^{n}F^{ij}W_{li}W_{lj}+\frac{1}{2}\sum_{i=1}^{n}F^{ii}\sum_{l=1}^{n}u_{1l}^{2}\nonumber \\
\ge & -3u_{11}\left|\frac{\sigma_{k}(W)}{\sigma_{k-1}(W)}\right|+\frac{mk}{n-k+1}\frac{\sigma_{k}(W)^{2}}{\sigma_{k-1}(W)^{2}}+m\frac{(k-1)}{n}\frac{\nu\sigma_{1}(W)}{\sigma_{k-1}(W)}\nonumber \\
 & +\frac{1}{2}(\frac{n-k+1}{k})u_{11}^{2}\nonumber \\
\ge &(\sqrt{2m }-3)\left|\frac{\sigma_{k}(W)}{\sigma_{k-1}(W)}\right|+ m\frac{(k-1)}{n}\frac{\nu\sigma_{1}(W)}{\sigma_{k-1}(W)}\nonumber \\
\ge & m\frac{(k-1)}{n}\frac{\nu\sigma_{1}(W)}{\sigma_{k-1}(W)}.\label{eq:inequality 1}
\end{align}

Now  at $(x_{0},t_{0})$, since we assume that {$u_{11}\ge 104nk|\nabla u|^{2}$}, 
using the definition of $F$ in \eqref{eq:final 1-1} and  \eqref{eq:inequality 1}
we have 
\begin{align*}
0\ge & -(8u_{1}^{2}F+2u_{11}F)-4m|\nabla u|^{2}F +\sum_{i=1}^{n}F^{ii}\sum_{l=1}^{n}u_{1l}^{2}+2m\sum_{i,j,l=1}^{n}F^{ij}W_{li}W_{lj}\\
 & -\frac{2k\nu}{\sigma_{k-1}(W)}u_{11}(1+\frac{1}{8k})-CG(1+|\nabla u|^2)\sum_{i=1}^{n}F^{ii}\\
\ge & -(8u_{1}^{2}+2u_{11}+4m|\nabla u|^{2})\frac{\sigma_{k}(W)}{\sigma_{k-1}(W)}-\frac{\nu}{\sigma_{k-1}}(2k-1)u_{11}\\
 & +\sum_{i=1}^{n}F^{ii}\sum_{l=1}^{n}u_{1l}^{2}+2m\sum_{i,j,l=1}^{n}F^{ij}W_{li}W_{lj}-CG(1+|\nabla u|^2)\sum_{i=1}^{n}F^{ii}\\
\ge & -3u_{11}\left|\frac{\sigma_{k}(W)}{\sigma_{k-1}(W)}\right|-\frac{\nu}{\sigma_{k-1}}(2k-1)u_{11}+\sum_{i=1}^{n}F^{ii}\sum_{l=1}^{n}u_{1l}^{2}+2m\sum_{i,j,l=1}^{n}F^{ij}W_{li}W_{lj}-CG(1+\nabla u|^2)\sum_{i=1}^{n}F^{ii}\\
\ge & -\frac{2k\nu}{\sigma_{k-1}(W)}W_{11}+\frac{1}{2}\sum_{i=1}^{n}F^{ii}\sum_{l=1}^{n}u_{1l}^{2}+m\sum_{i,j,l=1}^{n}F^{ij}W_{li}W_{lj}\\
 & -CG(1+|\nabla u|^2)\sum_{i=1}^{n}F^{ii}+m\frac{(k-1)}{n}\frac{\nu\sigma_{1}(W)}{\sigma_{k-1}(W)},
\end{align*}
where the last inequality holds due to (\ref{eq:inequality 1}).

For $k\ge3$, since $W_{11}\le\sigma_{1}(W)$ due to $W\in \Gamma_2^+$ and we have 
\begin{align*}
0\ge & -\frac{2k\nu}{\sigma_{k-1}}W_{11}+\frac{1}{2}\sum_{i=1}^{n}F^{ii}\sum_{l=1}^{n}u_{1l}^{2}+m\sum_{i,j,l=1}^{n}F^{ij}W_{li}W_{lj}\\
 & -CG(1+|\nabla u|^2)\sum_{i=1}^{n}F^{ii}+m\frac{(k-1)}{n}\frac{\nu\sigma_{1}(W)}{\sigma_{k-1}(W)}\\
\ge & \frac{1}{2}\sum_{i=1}^{n}F^{ii}\sum_{l=1}^{n}u_{1l}^{2}-CG(1+|\nabla u|^2)\sum_{i=1}^{n}F^{ii}\\
\ge & \frac{G^2}{8}\sum_{i=1}^{n}F^{ii}-CG(1+|\nabla u|^2)\sum_{i=1}^{n}F^{ii},
\end{align*}
where we have used $u_{11}^{2}\ge G^2/4$. Thus, we have $G\le C(1+|\nabla u|^2)\le C$.

For $k=2$, by \eqref{k2} we have 
\begin{align*}
0\ge & -\frac{2k\nu}{\sigma_{1}(W)}W_{11}+\frac{1}{2}\sum_{i=1}^{n}F^{ii}\sum_{l=1}^{n}u_{1l}^{2}+m\sum_{i,j,l=1}^{n}F^{ij}W_{li}W_{lj}\\
 & -CG(1+|\nabla u|^2)\sum_{i=1}^{n}F^{ii}+m\frac{(k-1)}{n}\frac{\nu\sigma_{1}(W)}{\sigma_{k-1}(W)}
 \\
\ge & -\frac{2k\nu}{\sigma_{1}(W)}W_{11}+\frac{\nu m\sum_{i=1}^{n}W_{ii}^{2}}{\sigma_{1}(W)^{2}}+m\frac{(k-1)}{n}\nu\\
 & +\frac{1}{2}\sum_{i=1}^{n}F^{ii}\sum_{l=1}^{n}u_{1l}^{2}-CG(1+|\nabla u|^2)\sum_{i=1}^{n}F^{ii}+\frac{2m\sigma_{2}(W)^{2}}{(n-1)\sigma_{1}(W)^{2}}\\
\ge & \frac{1}{2}\sum_{i=1}^{n}F^{ii}\sum_{l=1}^{n}u_{1l}^{2}-CG(1+|\nabla u|^2)\sum_{i=1}^{n}F^{ii}\\
\ge & \frac{G^2}{8}\sum_{i=1}^{n}F^{ii}-CG(1+|\nabla u|^2)\sum_{i=1}^{n}F^{ii},
\end{align*}
where in the third inequality we have used that  $-\frac{2k\nu}{\sigma_{1}(W)}W_{11}+\frac{\nu m\sum_{i=1}^{n}W_{ii}^{2}}{\sigma_{1}(W)^{2}}+m\frac{(k-1)}{n}\nu\ge 0$.
Finally, we have $G\le C(1+|\nabla u|^2)\le C$ and complete the proof.
\end{proof}

\section{Proof of the main theorems}\label{Section:Proof of Main theorems}

\subsection{\texorpdfstring{$C^2$}{C2} estimates for \texorpdfstring{\eqref{eq:pertubed flow}}{(perturbed flow)}}

We now combine the second derivative estimates from the previous section with the gradient bound to obtain a uniform $C^{0}$ bound and hence a uniform $C^{2}$ bound.
Since flow \eqref{eq:pertubed flow} preserves $\int_{M}\sigma_{k-1}(g)\,dv_{g}$ (instead of the volume),  the $C^{0}$ estimate requires an additional argument.

\begin{thm}\label{cor:aprioriestimatesforperturbedflow}
On a locally conformally flat manifold $(M,g_{0})$, let $g=e^{-2u}g_{0}\in \mathcal{C}_{k-1}$ satisfy \eqref{eq:pertubed flow} for some $\rho>0$.
Then
\[
\|u\|_{C^{2}}\le C,
\]
where $C$ depends only on $(M,g_0)$ and is independent of $\epsilon$.
\end{thm}

\begin{proof}
We first prove a uniform upper bound for the volume $\operatorname{Vol}(g)$; together with $|\nabla u|\le C$, this implies a uniform lower bound for $u$.

Since $g\in\mathcal{C}_{k-1}$, the inequalities of Guan--Wang~\cite{GWDuke} give:
if $k-1<\frac{n}{2}$, then
\[
\int_{M}\sigma_{k-1}(g)\,dv_{g}\ge C\,\operatorname{Vol}(g)^{\frac{n-2(k-1)}{n}};
\]
if $k-1>\frac{n}{2}$, then
\[
\operatorname{Vol}(g)^{\frac{2(k-1)-n}{n}}\int_{M}\sigma_{k-1}(g)\,dv_{g}\le C_{s};
\]
and if $k-1=\frac{n}{2}$, then
\[
\mathscr{E}_{n/2}(g)\ge C_s\,\log\frac{ \operatorname{Vol}(g)}{\operatorname{Vol}(g_0)}.
\]
Because $\int_{M}\sigma_{k-1}(g)\,dv_g$ (or $\mathscr{E}_{n/2}(g)$ when $k-1=\frac{n}{2}$) is constant along \eqref{eq:pertubed flow} by Lemma~\ref{lem:fucntional increase and r constant pertubed flow}, these inequalities yield, in all cases,
\begin{equation}\label{volume is uniformly bounded}
\operatorname{Vol}(g)=\int_M e^{-nu}\,dv_{g_{0}}\le C.
\end{equation}
Consequently, using $|\nabla u|\le C$, we obtain
\begin{equation}\label{lowerbound of u}
u\ge -C.
\end{equation}

It remains to prove a uniform upper bound for $u$.

\noindent \emph{Case 1: $k-1<\frac{n}{2}$.}
By Theorem~\ref{thm:C2 estimates} and Lemma~\ref{lem:gradient est},
\begin{align*}
\int_M\sigma_{k-1}(g_0)\,dv_{g_0}
&=\int_M\sigma_{k-1}(g)\,dv_{g}\\
&=\int_{M}e^{(2(k-1)-n)u}\,\sigma_{k-1}\Bigl(\nabla^{2}u+du\otimes du-\frac{|\nabla u|^{2}}{2}g_{0}+S_{g_{0}}\Bigr)\,dv_{g_{0}}\\
&\le C\int_{M}e^{(2(k-1)-n)u}\,dv_{g_{0}}.
\end{align*}
Since $2(k-1)-n<0$, this implies a uniform upper bound for $u$.

\medskip

\noindent \emph{Case 2: $k-1>\frac{n}{2}$.}
Assume for contradiction that $u$ is unbounded above.
Along the flow we have $\sup_M\bigl(|\nabla u|+|\nabla^{2}u|\bigr)\le C$, hence $\inf_{M}u\to\infty$ as $t\to T^{*}$.
Define
$$
v(x,t):=u(x,t)-\inf_{M}u(\cdot,t),
$$
so that $\|v\|_{C^{2}(M)}\le C$.
Passing to a subsequence $t_j\to T^{*}$, we may assume $v(\cdot,t_j)\to v^{*}$ in $C^{1,\alpha}$ for any $0\le\alpha<1$, for some $C^{1,1}$ function $v^{*}$.
Let $g_{v}=e^{-2v}g_{0}$.
Then
\begin{align*}
\int_{M}\sigma_{k-1}(g_{v})\,dv_{g_{v}}
&=e^{-(2(k-1)-n)\inf_{M}u}\int_{M}\sigma_{k-1}(g)\,dv_{g}\\
&=\int_{M}\sigma_{k-1}(g_0)\,dv_{g_0}\,e^{-(2(k-1)-n)\inf_{M}u}\longrightarrow 0,
\qquad \text{as } t\to T^{*}.
\end{align*}
Since $g_{v}\in\mathcal{C}_{k-1}$ for $t\in[0,T^{*})$, it follows that $\sigma_{k-1}(g_{v^{*}})=0$ in the weak sense.
This contradicts a result in \cite{GWDuke}, which states that if $\mathcal{C}_{k-1}[g_0]$ is non-empty, then there is no $C^{1,1}$ metric $\tilde g\in[g_0]$ with $\sigma_{k-1}(\tilde g)=0$.

\medskip

\noindent \emph{Case 3: $k-1=\frac{n}{2}$.}
Recall
$$
\mathscr{E}_{n/2}(g)=-\int_{0}^{1}\int_{M}\sigma_{n/2}(g_{s})\,u\,dv_{g_{s}}\,ds,
\qquad g_s=e^{-2su}g_0.
$$
In this case we know that the flow preserves $\mathcal{E}_{n\slash 2}$.
Moreover, by the Gauss--Bonnet--Chern formula, $\int_{M}\sigma_{n/2}(g)\,dv_{g}=c_0>0$ is a fixed constant in the given conformal class.
Since we only need to show that $\inf u$ is bounded from above, we may assume $\inf _M u \ge 0 $.  Since $g_0, g\in \mathcal C_{k-1}[g_0]$,   it holds that  $\sigma_{n/2}(g_s)>0$  by the convexity, and then,
\eq{
\mathcal{E}_{n\slash 2} (g) \le -c_0\inf u .
}
Hence $\inf u$ is bounded from above. 
Together with $|\nabla u|\le C$, this gives a uniform $C^{0}$ bound for $u$.

Combining the $C^{0}$ bound with Theorem~\ref{thm:C2 estimates} and Lemma~\ref{lem:gradient est} completes the proof.
\end{proof}

\begin{rem}
The above argument shows that if $|\nabla^{2}u|+|\nabla u|^{2}$ is uniformly bounded,  and along the flow $\int_{M}\sigma_{k-1}(g)\,dv_g$ is constant for  or $\mathscr{E}_{n/2}(g)$ is constant for $k-1= \frac n2$, then $\operatorname{Vol}(g)$ is uniformly bounded from above and bounded away from $0$.
\end{rem}
\subsection{Positivity of \texorpdfstring{$\s_{k-1}$}{sigma k-1} along flows}

In this subsection we prove that the perturbed flow \eqref{eq:pertubed flow} preserves the positivity of $\sigma_{k-1}(g)$. Equivalently, together with Theorem \ref{cor:aprioriestimatesforperturbedflow}, the equation remains uniformly parabolic for as long as the solution exists.

\begin{lem}\label{prop: uniform parabolicity for perturbed flow}
Let $g$ satisfy \eqref{eq:pertubed flow}.
Then for any $\epsilon>0$, there exists a constant $C_0=C_{0}(\epsilon, g_0)>0$, independent of $T<T^{*}$, such that for all $t\in[0,T]$,
\[
\sigma_{k-1}(g)\ge C_{0}.
\]
\end{lem}
\begin{proof}
The argument follows the maximum-principle strategy of Guan--Wang~\cite{GWDuke}. Since we only assume $g(t)\in\mathcal{C}_{k-1}$, we introduce an auxiliary function adapted to this weaker information. 

Fix $T<T^{*}$ and define on $M\times[0,T]$
\[
H=\frac{\sigma_{k}(g)-\rho}{\sigma_{k-1}(g)}e^{\varphi(u)}.
\]
Here $\varphi(s)=e^{-As-B}$ for {$s\in\bigl[-\max_{M\times[0,T]}|u|,\max_{M\times[0,T]}|u|\bigr]$}, and $A,B>0$ are chosen so that
\[
\varphi'(s)=-Ae^{-As-B}<0
\qquad\text{and}\qquad
\varphi'(s)+\varphi''(s)-(\varphi'(s))^{2}>c_*>0.
\]
Let $(x_0,t_0)\in M\times[0,T]$ be a point where $H$ attains its minimum. Then
\begin{equation}\label{eq:x1}
\frac{\sigma_{k}(g)-\rho}{\sigma_{k-1}(g)}e^{\varphi(u)}
\ge \frac{\sigma_{k}(g)(x_{0})-\rho}{\sigma_{k-1}(g)(x_{0})}e^{\varphi(u)(x_{0})}.
\end{equation}
By the Maclaurin inequality and the uniform $C^{2}$ bound, \eqref{eq:x1} yields
\begin{align*}
C(n,k)\frac{\sigma_{k-1}^{\frac{k}{k-1}}(g)}{\sigma_{k-1}(g)}
\ge \frac{\sigma_{k}(g)}{\sigma_{k-1}(g)}
\ge \frac{\rho}{\sigma_{k-1}(g)}-\frac{C}{\sigma_{k-1}(g)(x_{0})}.
\end{align*}
Consequently,
\[
\frac{\rho}{\sigma_{k-1}(g)}\le \frac{C}{\sigma_{k-1}(g)(x_{0})}+C.
\]
Thus it suffices to show that $\sigma_{k-1}(g)(x_{0})\ge C_{0}$; once this is established, the preceding inequality implies a uniform lower bound $\sigma_{k-1}(g)\ge C_{1}>0$ on $M$.

{Without loss of generality, we assume $H(x_0,t_0)<0$. Otherwise, it follows again from the Maclaurin inequality at point $(x_0,t_0)$
\begin{align*}
C(n,k){\sigma_{k-1}^{\frac{k}{k-1}}(g)}
\ge {\sigma_{k}(g)}\ge \epsilon.
\end{align*}
At the same time, when $\sigma_{k}(g)(x_0,t_0)\ge 0$, we have from the Maclaurin inequality
$$
H(x_0,t_0)=O(1)-\frac{\epsilon e^{\varphi(u)(x_0,t_0)}}{\sigma_{k-1}(g)(x_0,t_0)}.
$$
When $\sigma_{k}(g)(x_0,t_0)< 0$, we infer
$$
-H(x_0,t_0)\ge\frac{\epsilon e^{\varphi(u)(x_0,t_0)}}{\sigma_{k-1}(g)(x_0,t_0)}.
$$
Therefore, we have
$$
-H(x_0,t_0)\ge \frac{\epsilon e^{\varphi(u)(x_0,t_0)}}{2\sigma_{k-1}(g)(x_0,t_0)}\ge \frac{C\epsilon }{2\sigma_{k-1}(g)(x_0,t_0)},
$$
provided $\sigma_{k-1}(g)(x_0,t_0)$ is sufficiently small.
}

We work in geodesic normal coordinates at $x_{0}$. We may assume that $H(x_{0},t_{0})$ is sufficiently negative and $\sigma_{k-1}(g)(x_0,t_0)$ is sufficiently small; otherwise the desired conclusion already follows.

Set
\begin{align*}
T^{ij}:= \frac{\partial}{\partial S_{ij}}\left(\frac{\sigma_{k}(g)-\rho}{\sigma_{k-1}(g)}\right)
= \frac{ \sigma_{k-1}(g)T_{k-1}(g)^{ij}-\sigma_{k}(g)T_{k-2}(g)^{ij} }{\sigma_{k-1}^{2}(g)}+\rho\frac{T_{k-2}(g)^{ij}}{\sigma_{k-1}(g)^{2}}.
\end{align*}
At $(x_{0},t_{0})$, using $\frac{d S_{ij}}{dt}=\nabla^g_{ij}u_t$, we compute
\begin{align}\label{eq:minimum principle}
\frac{dH}{dt}
= {}& \frac{d}{dt}\bigg(\frac{\sigma_{k}(g)-\rho}{\sigma_{k-1}(g)}\bigg)e^{\varphi(u)}+\varphi'u_{t}\frac{\sigma_{k}(g)-\rho}{\sigma_{k-1}(g)}e^{\varphi(u)}\nonumber\\
= {}& e^{\varphi(u)}\sum_{i,j}T^{ij}\nabla_{ij}^{g}\frac{\sigma_{k}(g)-\rho}{\sigma_{k-1}(g)}+2e^{\varphi(u)}u_{t}\frac{\sigma_{k}+\rho(k-1)}{\sigma_{k-1}}+\varphi'e^{\varphi(u)}u_{t}\frac{\sigma_{k}(g)-\rho}{\sigma_{k-1}(g)}\nonumber\\
= {}& \sum_{i,j}e^{\varphi(u)}T^{ij}\left((He^{-\varphi})_{ij}+u_{i}(He^{-\varphi})_{j}+u_{j}(He^{-\varphi})_{i}-\sum_{l}u_{l}(He^{-\varphi})_{l}\delta_{ij}\right)\nonumber\\
&+\varphi'e^{\varphi(u)}(\frac{\sigma_{k}(g)-\rho}{\sigma_{k-1}(g)}-s_{\rho})\frac{\sigma_{k}-\rho}{\sigma_{k-1}}
+2e^{\varphi(u)}(\frac{\sigma_{k}(g)-\rho}{\sigma_{k-1}(g)}-s_{\rho})\frac{\sigma_{k}+\rho(k-1)}{\sigma_{k-1}}.
\end{align}

Set
\[
P:=\sum_{i,j}T^{ij}\left((He^{-\varphi})_{ij}+u_{i}(He^{-\varphi})_{j}+u_{j}(He^{-\varphi})_{j}-\sum_{l}u_{l}(He^{-\varphi})_{l}\delta_{ij}\right).
\]

At $(x_0,t_0)$ we have $H_i=0$. A direct computation gives
\begin{align}
P
= & \sum_{i,j}T^{ij}e^{-\varphi}H_{ij}+\sum_{i,j}e^{-\varphi}HT^{ij}\left(-\varphi'u_{ij}-\varphi''u_{i}u_{j}+(\varphi')^{2}u_{i}u_{j}-2u_{i}u_{j}\varphi'+\varphi'|\nabla u|^{2}\delta_{ij}\right)\nonumber\\
= &\sum_{i,j} T^{ij}e^{-\varphi}H_{ij}+\sum_{i,j}e^{-\varphi}HT^{ij}\bigg(-\varphi'(S_{ij}-u_{i}u_{j}+\frac{|\nabla u|^{2}}{2}\delta_{ij}-S(g_{0})_{ij})-\varphi''u_{i}u_{j}\nonumber\\
 & +(\varphi')^{2}u_{i}u_{j}-2u_{i}u_{j}\varphi'+\varphi'|\nabla u|^{2}\delta_{ij}\bigg) \nonumber \\
= & \sum_{i,j}\left(T^{ij}e^{-\varphi}H_{ij}-\varphi'e^{-\varphi}HT^{ij}S_{ij}+\varphi'e^{-\varphi}HT^{ij}S(g_{0})_{ij}\right)+\sum_{i}e^{-\varphi}HT^{ii}\frac{|\nabla u|^{2}}{2}\varphi'\nonumber\\
 & +\sum_{i,j}e^{-\varphi}HT^{ij}u_{i}u_{j}(-\varphi'-\varphi''+(\varphi')^{2})
 \nonumber \\
\ge  &  \sum_{i,j}\left(T^{ij}e^{-\varphi}H_{ij}-\varphi'e^{-\varphi}HT^{ij}S_{ij}+\varphi'e^{-\varphi}HT^{ij}S(g_{0})_{ij}\right)+\sum_{i}e^{-\varphi}HT^{ii}\frac{|\nabla u|^{2}}{2}\varphi',
 \label{simple expression of P}
\end{align}
where in the last inequality we used the choice of $\varphi$, namely $-\varphi'-\varphi''+(\varphi')^{2}<0$ and $H<0$.

By \eqref{sigmaproperty1} and \eqref{sigmaproperty2}, 
\begin{equation}
\sum_{i,j}T^{ij}S_{ij}=\frac{\sigma_{k}(g)+\rho(k-1)}{\sigma_{k-1}(g)}.\label{eq:trace S}
\end{equation}
By  Lemma \ref{lem:Garding}, we have
\begin{align}
\sum_{i,j}T^{ij}S(g_{0})_{ij}= & \sum_{i,j}\frac{\partial}{\partial S_{ij}}(\frac{\sigma_{k}(g)}{\sigma_{k-1}(g)})S(g_{0})_{ij}-\sum_{i,j}\rho\frac{\partial}{\partial S_{ij}}(\frac{1}{\sigma_{k-1}(g)})S(g_{0})_{ij}\nonumber\\
\ge & e^{2u}\frac{\sigma_{k}(g_{0})}{\sigma_{k-1}(g_{0})}+\rho(k-1)e^{2u}(\frac{1}{\sigma_{k-1}(g)})^{2}\frac{\sigma_{k-1}(g_{0})^{\frac{1}{k-1}}}{\sigma_{k-1}(g)^{\frac{1}{k-1}-1}}\nonumber \\
\ge & -C+\rho(k-1)e^{2u}(\frac{1}{\sigma_{k-1}(g)})^{1+\frac{1}{k-1}}\sigma_{k-1}(g_{0})^{\frac{1}{k-1}}.\label{eq:positive lower bound}
\end{align}

Substituting \eqref{eq:trace S} and \eqref{eq:positive lower bound} into \eqref{simple expression of P} yields
\begin{align*}
P%
\ge {}&\sum_{i,j} e^{-\varphi}T^{ij}H_{ij}-\varphi'e^{-\varphi}H\frac{\sigma_{k}(g)+\rho(k-1)}{\sigma_{k-1}(g)}\\
&+\varphi'e^{-\varphi}H\left(-C+\rho(k-1)e^{2u}\left(\frac{1}{\sigma_{k-1}(g)}\right)^{1+\frac{1}{k-1}}\sigma_{k-1}(g_{0})^{\frac{1}{k-1}}\right).
\end{align*}
Recall that $|s_{\epsilon}=\frac{\int_M (\sigma_{k}(g)-\epsilon)\,dv_{g}}{\int_M \sigma_{k-1}(g)\,dv_{g}}|\le C$ by Theorem \ref{cor:aprioriestimatesforperturbedflow} and Lemma \ref{lem:fucntional increase and r constant pertubed flow}. 
Hence we have at $(x_0,t_0)$
{
\begin{align}
0\ge {}& \frac{dH}{dt}-\sum_{i,j}T^{ij}H_{ij}
\nonumber\\
\ge {}&-\varphi'H\frac{\sigma_{k}(g)+\rho(k-1)}{\sigma_{k-1}(g)}
+\varphi'H\left(-C+\rho(k-1)e^{2u}\left(\frac{1}{\sigma_{k-1}(g)}\right)^{1+\frac{1}{k-1}}\sigma_{k-1}(g_{0})^{\frac{1}{k-1}}\right)
\nonumber\\
&+\varphi'e^{\varphi(u)}\left(\frac{\sigma_{k}(g)-\rho}{\sigma_{k-1}(g)}-s_{\rho}\right)\frac{\sigma_{k}(g)-\rho}{\sigma_{k-1}(g)}
+2e^{\varphi(u)}\left(\frac{\sigma_{k}(g)-\rho}{\sigma_{k-1}(g)}-s_{\rho}\right)\frac{\sigma_{k}(g)+\rho(k-1)}{\sigma_{k-1}(g)}
\nonumber\\
\ge {}& -C\frac{1}{\sigma_{k-1}(g)^2}
+\varphi'e^{2u}H\rho(k-1)\left(\frac{1}{\sigma_{k-1}(g)}\right)^{1+\frac{1}{k-1}}\sigma_{k-1}(g_{0})^{\frac{1}{k-1}}
\nonumber\\
\ge {}& -C\frac{|H(x_{0},t_{0})|}{\sigma_{k-1}(g)}
+\varphi'e^{2u}H\rho(k-1)\left(\frac{1}{\sigma_{k-1}(g)}\right)^{1+\frac{1}{k-1}}\sigma_{k-1}(g_{0})^{\frac{1}{k-1}}
\nonumber
\end{align}
where in the last two inequalities  we have used }Theorem~\ref{cor:aprioriestimatesforperturbedflow} and the assumption that $|H(x_0,t_0)|$ is sufficiently large, thus the remaining terms are all absorbed into  $-C\frac{|H(x_{0},t_{0})|}{\sigma_{k-1}(g)}$. 
It follows  from $\varphi 'e^{2u}<-c_1<0$ that 
\begin{align*}
0\ge {}& |H|\left(-\frac{C}{\sigma_{k-1}}+c_{0}\rho\sigma_{k-1}^{-\frac{k}{k-1}}\right),\qquad \text{for some constant }c_0>0.
\end{align*}
Therefore there exists $C_{0}>0$ such that $\sigma_{k-1}(g)(x_{0},t_0)\ge C_{0}$.

\end{proof}

\subsection{Convergence of flow \texorpdfstring{\eqref{eq:pertubed flow}}{(perturbed flow)}}
\begin{thm}\label{thm:long-time existence for pertubed flow}
Let $(M^{n},g_{0})$ be locally conformally flat and let $k\ge 2$.
Assume that $g_{0}\in \mathcal C_{k-1}$.
For every fixed $\rho>0$, the flow \eqref{eq:pertubed flow} with the initial metric $g_0$ exists for all time and converges to a metric $g^{\rho}_{\infty}\in \mathcal C_{k-1}[g_0]$ satisfying
\begin{equation}\label{limiting equationfinial}
\frac{\sigma_{k}(g^{\rho}_{\infty})-\rho}{\sigma_{k-1}(g^{\rho}_{\infty})}-s_{\rho}=0\quad\text{on }M,
\end{equation}
for some constant $s_\epsilon$.
Moreover, there exists a constant $C>0$, independent of $\rho$, such that
\begin{equation}\label{volume}
\operatorname{Vol}(g^{\rho}_{\infty})\le C.
\end{equation}

If we assume additionally that $Y_{k,k-1}(M,[g_0])>0$, for $k<\frac n 2$, or $\int_M \sigma_k(g_0)\,dv_{g_0}>0$ for $k\ge \frac n 2$, 
then for sufficiently small $\rho$ we have $g^{\rho}_{\infty}\in \mathcal C_k$.
\end{thm}

\begin{proof}
By Theorem~\ref{cor:aprioriestimatesforperturbedflow} and Lemma~\ref{prop: uniform parabolicity for perturbed flow}, with the standard argument as \cite{GWDuke}, the flow exists for all time and converges to a metric $g_{\infty}^{\rho}\in\mathcal C_{k-1}$ satisfying \eqref{limiting equationfinial} and \eqref{volume}. It remains to show the last statement.

It suffices to show
$s_{\rho}\ge 0$, since  \eqref{limiting equationfinial} implies that $\sigma_k(g^{\rho}_{\infty })>0$.

If $k<\frac n 2$, then by the definition of $Y_{k,k-1}(M,[g_0])$,
\begin{align*}
\int_M \sigma_k(g^{\rho}_{\infty })\,dv_{g^{\rho}_{\infty }}
&\ge Y_{k,k-1}(M,[g_0])\left(\int_M \sigma_{k-1}(g^{\rho}_{\infty })\,dv_{g^{\rho}_{\infty }}\right)^{\frac{n-2k}{n-2(k-1)}}\\
&=Y_{k,k-1}(M,[g_0])\left(\int_M \sigma_{k-1}(g_0)\,dv_{g_0}\right)^{\frac{n-2k}{n-2(k-1)}}.
\end{align*}
Therefore, for $\rho$ sufficiently small,
\[
 s_{\rho}=\frac{\int_M(\sigma_{k}(g^{\rho}_{\infty })-\rho)\,dv_{g^{\rho}_{\infty }}}{\int_M\sigma_{k-1}(g^{\rho}_{\infty })\,dv_{g^{\rho}_{\infty }}}>0,
\]
since $\operatorname{Vol}(g_\infty^\epsilon)$ is bounded.
Consequently, $g^{\rho}_{\infty }\in \mathcal C_{k}$ by \eqref{limiting equationfinial}.

If $k>\frac n 2$, then by the monotonicity in Lemma~\ref{lem:fucntional increase and r constant pertubed flow},
\[
\int_M \sigma_k(g^{\rho}_{\infty })\,dv_{g^{\rho}_{\infty }}+\frac{(2k-n)\rho}{n}\operatorname{Vol}(g^{\rho}_{\infty })
\ge \int_M \sigma_k(g_0)\,dv_{g_0}+\frac{(2k-n)\rho}{n}\operatorname{Vol}(g_0).
\]
Thus, for $\rho$ sufficiently small, if $\int_M \sigma_k(g_0)\,dv_{g_0}>0$ then $s_{\rho}>0$ by \eqref{volume}, and hence $g^{\rho}_{\infty }\in \mathcal C_{k}$.

Finally, if $k=\frac n2$, then $\int_M \sigma_{n/2}(g_0)\,dv_{g_0}$ is constant and again $s_{\rho}>0$ for $\rho$ sufficiently small.
This completes the proof.
\end{proof}

With the help of Theorem~\ref{thm:long-time existence for pertubed flow}, we can now establish Theorem~\ref{thm:minimizer in weak cone}.

\begin{proof}[Proof of Theorem~\ref{thm:minimizer in weak cone}] Assume $Y_{k,k-1}(M,[g_0])>0$.

First consider the case $k<\frac n 2$. 
Choose a sequence $g_i\in \mathcal C_{k-1}[g_0]$ such that
\[
\lim_{i\to\infty}\frac{\int_M \sigma_{k}(g_i)\,dv_{g_i}}{\left(\int_M \sigma_{k-1}(g_i)\,dv_{g_i}\right)^{\frac{n-2k}{n-2(k-1)}}}=Y_{k,k-1}(M,[g_0]).
\]
For each $i$, run the modified  flow \eqref{eq:pertubed flow} with initial metric $g_i$.
By Theorem~\ref{thm:long-time existence for pertubed flow}, the flow converges to a limit metric $g^{\rho}_{i,\infty}$.
By Lemma~\ref{lem:fucntional increase and r constant pertubed flow},
\eq{\label{eq_x}
J^k_\rho(g_i)\ge J^k_\rho(g^{\rho}_{i,\infty}),
\qquad J^k_\rho(g)=\int_M\sigma_{k}(g)\,dv_{g}+\frac{(2k-n)\rho}{n}\operatorname{Vol}(g).
}
Wlog we may assume $\int_M\s_{k-1} (g_i) dv_{g_i}=1.$
\eqref{eq_x} implies
\begin{align*}
\int_M \sigma_{k}(g_i)\,dv_{g_i}+\frac{(2k-n)\rho}{n}\operatorname{Vol}(g_i)
\ge 
{\int_M \sigma_{k}(g^{\rho}_{i,\infty})\,dv_{g^{\rho}_{i,\infty}}+\frac{(2k-n)\rho}{n}\operatorname{Vol}(g^{\rho}_{i,\infty})}
\end{align*}

For $\rho$ sufficiently small, Theorem~\ref{thm:long-time existence for pertubed flow} yields $g^{\rho}_{i,\infty}\in \mathcal C_{k}$ and $\rho\operatorname{Vol}(g^{\rho}_{i,\infty})\to0$ as $\rho\to0$.
Hence
\[
{\int_M \sigma_{k}(g_i)\,dv_{g_i}}
\ge \lim_{\rho\to0}{\int_M \sigma_{k}(g^{\rho}_{i,\infty})\,dv_{g^{\rho}_{i,\infty}}}
\ge \bar Y_{k,k-1}(M,[g_0]).
\]
Thus $Y_{k,k-1}(M,[g_0])\ge \bar Y_{k,k-1}(M,[g_0])$.
Since $\mathcal C_{k}\subset \mathcal C_{k-1}$, we have trivially $Y_{k,k-1}(M,[g_0])\le \bar Y_{k,k-1}(M,[g_0])$.
Therefore $Y_{k,k-1}(M,[g_0])=\bar Y_{k,k-1}(M,[g_0])$ for $k<\frac n2$.

\medskip

For $k>\frac n2$, since $\mathcal C_{k}\subset \mathcal C_{k-1}$, we have $Y_{k,k-1}(M,[g_0])\ge \bar Y_{k,k-1}(M,[g_0])$ by definition.
Take any $g_*\in \mathcal C_{k-1}$ with $\int_M \sigma_{k}(g_*)\,dv_{g_*}>0$ and run \eqref{eq:pertubed flow} from $g_*$.
By Theorem~\ref{thm:long-time existence for pertubed flow} and Lemma~\ref{lem:fucntional increase and r constant pertubed flow}, there exists $g^{\rho}_{*,\infty}\in\mathcal C_{k}$ for $\rho$ sufficiently small and $\rho\operatorname{Vol}(g^{\rho}_{*,\infty})\to0$. Then  a similar argument as above implies
$Y_{k,k-1}(M,[g_0])\le \bar Y_{k,k-1}(M,[g_0])$, and hence equality holds. The achievement of $\bar Y_{k,k-1}(M,[g_0])$ was proved in \cite{GWDuke}.

We leave the proof of the statement $\mathcal{C}_k[g_0]\not =\emptyset$ implies $Y_{k,k-1} (M,[g_0])>0$ in the next section.
\end{proof}


With Theorem~\ref{thm:minimizer in weak cone} established, it remains to complete the proof of Theorem~\ref{thm:generalized Sobolev inequality-high dim} by treating the remaining cases $k=\frac n2$ and $k=\frac n2+1$.

\begin{proof}[Proof of Theorem~\ref{thm:generalized Sobolev inequality-high dim}]
Let $g_0=g_{\mathbb S^n}\in\mathcal C_{k-1}$ and it  holds that $\int_M \sigma_{k}(g_0)\,dv_{g_0}>0$ for $k\ge \frac n 2$.

Given any $g\in \mathcal C_{k-1}[g_0]$, run the flow \eqref{eq:pertubed flow} with initial metric $g$.
By Theorem~\ref{thm:long-time existence for pertubed flow}, the flow converges to $g^{\rho}_{\infty}\in \mathcal C_k[g_0]$ with
$\int_M \sigma_{k-1}(g)\,dv_g=\int_M \sigma_{k-1}(g^{\rho}_{\infty})\,dv_{g^{\rho}_{\infty}}$
and $\rho\operatorname{Vol}(g^{\rho}_{\infty})\to 0$ as $\rho\to 0$.

\medskip
\noindent\emph{Case $k=\frac n 2$.}
By the monotonicity in Lemma~\ref{lem:fucntional increase and r constant pertubed flow} and the Sobolev inequality in \cite[Theorem~1(C)]{GWDuke} applied to $g^{\rho}_{\infty}$, for $\rho$ sufficiently small we have
\begin{align*}
\mathscr{E}_{n/2}(g)
&\ge \mathscr{E}_{n/2}(g^{\rho}_{\infty})-\frac{\rho}{n}\operatorname{Vol}(g^{\rho}_{\infty})\\
&\ge C_s\left(\log \int_M \sigma_{\frac n 2-1}(g^{\rho}_{\infty})\,dv_{g^{\rho}_{\infty}}-\log \int_M \sigma_{\frac n 2-1}(g_{0})\,dv_{g_{0}}\right)-\frac{\rho}{n}\operatorname{Vol}(g^{\rho}_{\infty})\\
&= C_s\left(\log \int_M \sigma_{\frac n 2-1}(g)\,dv_g-\log \int_M \sigma_{\frac n 2-1}(g_0)\,dv_{g_0}\right)-\frac{\rho}{n}\operatorname{Vol}(g^{\rho}_{\infty}),
\end{align*}
where
$$
C_s=\frac{1}{2}C_{MT}=\frac{1}{2}\int_M \sigma_{\frac n 2}(g_0)\,dv_{g_0}=\frac{1}{2}\frac{\omega_n}{2^{n/2}}\binom{n}{n/2}.
$$
Letting $\rho\to 0$ gives
\[
\mathscr{E}_{n/2}(g)\ge C_s\left(\log \int_M \sigma_{\frac n 2-1}(g)\,dv_g-\log \int_M \sigma_{\frac n 2-1}(g_{0})\,dv_{g_{0}}\right).
\]

\medskip
\noindent\emph{Case $k=\frac n 2+1$.}
Applying the Moser--Trudinger inequality in \cite[Theorem~1(C)]{GWDuke} to $g^{\rho}_{\infty}\in \mathcal C_{\frac n 2+1}$, for $\rho$ sufficiently small we obtain
\begin{align}
J^k_\rho(g)
&=\int_M\sigma_{k}(g)\,dv_{g}+\frac{(2k-n)\rho}{n}\operatorname{Vol}(g)
\nonumber\\
&\le\int_M\sigma_{k}(g^{\rho}_{\infty})\,dv_{g^{\rho}_{\infty}}+\frac{(2k-n)\rho}{n}\operatorname{Vol}(g^{\rho}_{\infty})
\nonumber\\
&\le e^{-\frac{2\mathscr{E}_{n/2}(g)}{C_{MT}}}\int_M \sigma_{\frac n 2+1}(g_0)\,dv_{g_0}+\frac{(2k-n)\rho}{n}\operatorname{Vol}(g^{\rho}_{\infty}).
\end{align}
Letting $\rho\to 0$ yields
\[
\int_M\sigma_{\frac n 2+1}(g)\,dv_{g}
\le e^{-\frac{2\mathscr{E}_{n/2}(g)}{C_{MT}}}\int_M \sigma_{\frac n 2+1}(g_0)\,dv_{g_0}.
\]

Finally, combining the cases $l=k-1$ and a simple induction shows that the general case for $k$ and $l\le k-2$ follows from the Sobolev inequality in \cite[Theorem~1]{GWDuke}. See also the next Remark.
\end{proof}

\begin{remark}
    \label{iteration}
    One can iterate the inequality for $l=k-1$ to $l=k-2$ as follows. For example, for $k<\frac{n}{2}$,  using
    \eq{
    \frac{\int_M \sigma_k(g)\, dv_g}{\left(\int_M \sigma_{k-2}(g)\, dv_g\right)^{\frac{n-2 k}{n-2(k-2)}}}
    = 
    \frac{\int_M \sigma_k(g)\, dv_g}{\left(\int_M \sigma_{k-1}(g)\, dv_g\right)^{\frac{n-2 k}{n-2(k-1)}}}
   \left\{\frac{\int_M \sigma_{k-1}(g)\, dv_g}{\left(\int_M \sigma_{k-2}(g)\, dv_g\right)^{\frac{n-2 (k-1)}{n-2(k-2)}}} 
\right\}^{\frac {n-2k}{n-2(k-1)} }
    }
    we have
    \eq{
\frac{\int_M \sigma_k(g)\, dv_g}{\left(\int_M \sigma_{k-2}(g)\,  dv_g\right)^{\frac{n-2 k}{n-2(k-2)}}}
  \ge Y_{k,k-1} (M)Y_{k-1, k-2}(M) ^{\frac {n-2k}{n-2(k-1)} }.
  }
  If $M= \S^n$, one can show that 
   \eq{Y_{k,k-1} (\S^n )Y_{k-1, k-2}(\S^n) ^{\frac {n-2k}{n-2(k-1)} }=Y_{k,k-2}(\S^n),
   }
   since all minimizers of $Y_{k,k-1}$, $Y_{k,k-2}$ and $Y_{k-1, k-2}$ are achieved by the same metric, the round metric. Then we obtain the sharp inequality for $l=k-2$.
   On a manifold other than the round sphere, this is not the case. We have only
\eq{Y_{k,k-1} (M)Y_{k-1, k-2}(M) ^{\frac {n-2k}{n-2(k-1)} } \le Y_{k,k-2}(M).
   } It could happen that minimizers of these three Yamabe constants cannot be achieved by the same metric.  Therefore, the iteration method does not apply to obtain the sharp inequality, and we need to consider flow \eqref{flow for k-l}.
   \end{remark}

%


\section{Positivity of \texorpdfstring{$Y_{k,k-1}(M,[g_0])$}{Yk,k-1(M,[g0])}}\label{Section: Appendix}

Let $k< \frac n2$ and let $(M,g_0)$ be a locally conformally flat manifold with $g_0\in \mathcal C_{k-1}$. In this section we study the sign of the conformal invariant $Y_{k,k-1}(M,[g_0])$. Recall that
\[
Y_{k, k-1}\left(M,\left[g_0\right]\right)
= \inf_{g\in \mathcal C_{k-1}[g_0]}\frac{\int_M \sigma_k(g)\, dv_g}{\left(\int_M \sigma_{k-1}(g)\, dv_g\right)^{\frac{n-2 k}{n-2(k-1)}}}.
\]

We begin by showing

\begin{proposition}\label{prop6.1}
    Let $k< \frac n2$ and let $(M,g_0)$ be a locally conformally flat manifold with $g_0\in \mathcal C_{k-1}$. Then
    \eq{
    Y_{k, k-1}\left(M,\left[g_0\right]\right) > -\infty,
    }
    in particular, $Y_{k,k-1}$ is well-defined.
\end{proposition}

This finiteness statement is not obvious; compare the examples in the appendix below. We defer the proof to the end of this section.

    By Proposition~\ref{prop6.1} there are three possibilities, as in the classical Yamabe problem: (i) $Y_{k,k-1}>0$, (ii) $Y_{k,k-1}=0$, and (iii) $Y_{k,k-1}<0$. It is natural to ask whether the infimum in the definition of $Y_{k,k-1}$ is always achieved. The answer is yes, at least in $\overline{\mathcal{C}}_{k-1}$ in a viscosity sense. 


\begin{thm}\label{Y_quo}
On a locally conformally flat manifold $(M,g_0)$ with $g_0\in \mathcal C_{k-1}[g_0]$ for $k<\frac n 2$, exactly one of the following three cases occurs:
\begin{itemize}
    \item[(i)] If $Y_{k,k-1}(M,[g_0])>0$, then there exists a smooth metric $g\in\mathcal C_k[g_0]$ such that $\frac{\s_k(g)}{\s_{k-1}(g)}=1$.
    \item[(ii)] {If $Y_{k,k-1}(M,[g_0])=0$, then there exists a $C^{1,1}$ metric $g\in \overline{\mathcal C}_{k-1}[g_0]$ such that $\s_k(g)=0$.
    \item[(iii)] If $Y_{k,k-1}(M,[g_0])<0$, then there exists a $C^{1,1}$ metric $g\in\overline{\mathcal{C}}_{k-1}[g_0]$ such that ${\s_k(g)}=-{\s_{k-1}(g)}$}.
\end{itemize}
\end{thm}

\begin{proof}
When $k=2$, the result was proved in \cite{GLWJDG} without the locally conformally flat assumption. The main ideas of \cite{GLWJDG} apply here, so we only sketch the argument.

The proof in Section~\ref{Section:Proof of Main theorems} yields that for each $\rho>0$, there exists a solution $g_\rho$ with $\int \sigma_{k-1}(g_{\epsilon})=1$ to 
\eq{\label{limiting equation}
\frac{\s_k(g_\rho)-\rho}{\s_{k-1}(g_\rho)}=s_\rho=\frac{\int \sigma_{k}(g_{\epsilon})-\epsilon Vol(g_{\epsilon})}{\int \sigma_{k-1}(g_{\epsilon})}.}
Letting $\rho\to 0$, we obtain three cases: (i) $s_0>0$, (ii) $s_0=0$, and (iii) $s_0<0$.
 Case (i) was considered in the previous section.
Due to the estimates given in the previous section,  for the remaining cases one shows that $g_\rho$ converges to a {$C^{1,1}$} metric $g$ of
\eq{\label{eq_q} \s_k(g)=s_0\s_{k-1}(g).}
It is straightforward to see that these three cases correspond to the three possible signs of $Y_{k,k-1}$. Note that for $s_0<0$, the solution is unique.
For the reader's convenience, we provide a proof for the case $Y_{k,k-1}=0$. 

Since $Y_{k,k-1}(M)=0$, we can take a sequence of metrics $g_{i}\in\mathcal{C}_{k-1}$
such that $\int\sigma_{k}(g_{i})dv_{g_{i}}\rightarrow0$ and $\int\sigma_{k-1}(g_{i})=1$. 
Given an initial date $g_{i}$, consider the flow
 \begin{equation}
 u_t=-\frac{1}{2}g^{-1}\frac{\partial g}{\partial t}
 =\frac{\sigma_{k}(g)-\frac 1j}{\sigma_{k-1}(g)}-s_{j}
\quad\text{on }M,
\label{eq:pertubed flow-1/j}
\end{equation}
 where
\[
 s_{j}=\frac{\int_M (\sigma_{k}(g)-\frac 1j)\,dv_{g}}{\int_M \sigma_{k-1}(g)\,dv_{g}}.
\]

By Theorem \ref{thm:long-time existence for pertubed flow} and Lemma \ref{lem:fucntional increase and r constant pertubed flow},
there exists a limiting metric $g_{i,\infty}^{j}\in\mathcal{C}_{k-1}[g_{i}]=\mathcal{C}_{k-1}[g_{0}]$
such that
\begin{equation}
\frac{\sigma_{k}(g_{i,\infty}^{j})-\frac{1}{j}}{\sigma_{k-1}(g_{i,\infty}^{j})}=\frac{\int\sigma_{k}(g_{i,\infty}^{j})dv_{g_{i,\infty}^{j}}-\frac{1}{j}Vol(g_{i,\infty}^{j})}{\int\sigma_{k-1}(g_{i,\infty}^{j})dv_{g_{i,\infty}^{j}}}\label{eq:limiting equation on M}
\end{equation}
with $\int\sigma_{k-1}(g_{i,\infty}^{j})dv_{g_{i,\infty}^{j}}=1$.
Since $\int\sigma_{k-1}(g_{i,\infty}^{j})dv_{g_{i,\infty}^{j}}=1$ and
$\int\sigma_{k-1}(g_{i})=1$, by the Sobolev inequalities in \cite{GWDuke}, we
have 
\begin{equation}
Vol(g_{i})\le C,\quad Vol(g_{i,\infty}^{j})\le C.\label{eq:upper bound of volume}
\end{equation}
 By the monotonicity of the functional, see Lemma \ref{lem:fucntional increase and r constant pertubed flow},
\[
\int\sigma_{k}(g_{i,\infty}^{j})dv_{g_{i,\infty}^{j}}+\frac{2k-n}{n}\frac{1}{j}Vol(g_{i,\infty}^{j})=J(g_{i,\infty}^{j})\le J(g_{i})=\int\sigma_{k}(g_{i})dv_{g_{i}}+\frac{2k-n}{n}\frac{1}{j}Vol(g_{i}).
\]
Hence,
\begin{equation}
\int\sigma_{k}(g_{i,\infty}^{j})dv_{g_{i,\infty}^{j}}\le\int\sigma_{k}(g_{i})dv_{g_{i}}+\frac{2k-n}{n}\frac{1}{j}Vol(g_{i})-\frac{2k-n}{n}\frac{1}{j}Vol(g_{i,\infty}^{j}),\label{eq:smallness of intgeral of sigmak}
\end{equation}
and
\begin{align*}
\frac{\sigma_{k}(g_{i,\infty}^{j})-\frac{1}{j}}{\sigma_{k-1}(g_{i,\infty}^{j})} & =\int\sigma_{k}(g_{i,\infty}^{j})dv_{g_{i,\infty}^{j}}-\frac{1}{j}Vol(g_{i,\infty}^{j})\\
 & \le\int\sigma_{k}(g_{i})dv_{g_{i}}+\frac{2k-n}{n}\frac{1}{j}Vol(g_{i})-\frac{2k-n}{n}\frac{1}{j}Vol(g_{i,\infty}^{j})-\frac{1}{j}Vol(g_{i,\infty}^{j})\\
 & \le\int\sigma_{k}(g_{i})dv_{g_{i}}+\frac{2k-n}{n}\frac{1}{j}Vol(g_{i})-\frac{2k}{n}\frac{1}{j}Vol(g_{i,\infty}^{j})\\
 & <\int\sigma_{k}(g_{i})dv_{g_{i}}\quad\text{for }k<\frac{n}{2}.
\end{align*}

If $M$ is conformal to the sphere, then by the method of moving sphere, we know that $g_{i,\infty}^{j}$ can only be a standard sphere
metric up to conformal diffeomorphism. See \cite{LLActa,ViacTams} for example.
From $\int\sigma_{k-1}(g_{i,\infty}^{j})dv_{g_{i,\infty}^{j}}=1$,
we know that 
\[
\int\sigma_{k}(g_{i,\infty}^{j})dv_{g_{i,\infty}^{j}}=\binom{n}{k}2^{-k}\omega_{n}\left(\binom{n}{k-1}2^{-(k-1)}\omega_{n}\right)^{-\frac{n-2k}{n-2k+2}},
\]
which contradicts to (\ref{eq:smallness of intgeral of sigmak}) when 
$i,j$ are sufficiently large. So $M$ is not conformal to the sphere.

If the locally conformally flat manifold $M$ is not conformal to
sphere, then by the standard argument of moving planes, for example,
see \cite{LLActa, GLWMathZ}, we know that for $g_{i,\infty}^{j}=e^{-2u_{i,\infty}^{j}}g_{0}$,
it holds that 
\begin{equation}\label{gradient 1}
    |\nabla_{g_{0}}u_{i,\infty}^{j}|\le C,
\end{equation}
where $C$ is independent of $i,j$.
By (\ref{eq:upper bound of volume}), we get that 
\begin{equation}\label{gradient 2}
    u_{i,\infty}^{j}\ge-C.
\end{equation}
Then, by a priori estimates \cite{GWIMRN,SChenIMRN} and \eqref{gradient 1} and \eqref{gradient 2} we know that
\[
|\nabla_{g_{0}}^{2}u_{i,\infty}^{j}|\le C
\]
where $C$ is independent of $i,j$. 
Thus, there exists a subsequence of  $g_{i,\infty}^{j}$ converging
to $g_{\infty}$ in $C^{1,\alpha}$ for any $0<\alpha<1$ such that
$\sigma_{k}(g_{\infty})=0$ in the viscosity sense since $Y_{k,k-1}(M, [g])=0.$

\end{proof}

The ``if'' direction in the last statement of Theorem~\ref{thm:minimizer in weak cone} follows from Theorem~\ref{Y_quo}. Indeed, if $Y_{k,k-1}(M,[g_0])\le 0$, then Theorem~\ref{Y_quo} provides a metric $g\in \overline{\mathcal C}_{k-1}[g_0]$ with either $\sigma_k(g)=0$ or $\frac{\s_k(g)}{\s_{k-1}(g)}=-1$. This contradicts the non-emptiness of $\mathcal C_k[g_0]$ by the maximum principle; see \cite{GLWJDG}.


We now begin to show Proposition \ref{prop6.1}.
Inspired by Proposition 1.1 in Han \cite{Han}, we provide the following formula,
which will be used for the lower bound of $\mathcal{F}_{k}$ for $k<\frac{n}{2}$ and is of interest itself.
\begin{prop}\label{thm:intsigmak}
Let $(M^{n},g_{0})$ be locally conformally flat and set $g=e^{-2w}g_{0}$. Then
\begin{align}
k\sigma_{k}(g)
&= (n-2k)\sum_{a=1}^{k}\frac{\sigma_{k-a}(g)}{2^{a}}\,|\nabla w|_{g}^{2a}\label{eq:intsigmak}
 +\sum_{a=1}^{k}\frac{1}{2^{a-1}}\,(T_{k-a}(S_{g}))_{j}^{i}(S_{g_{0}})_{i}^{j}\,|\nabla w|_{g}^{2(a-1)}\nonumber \\
&\quad +\sum_{a=2}^{k}\frac{a-1}{2^{a-2}}\,(T_{k-a}(S_{g}))_{j}^{i}(S_{g_{0}})_{l}^{j}w^{l}w_{i}\,|\nabla_{g}w|^{2(a-2)}\nonumber\\
&\quad +\sum_{a=0}^{k-1}\nabla^j\Bigl(\frac{1}{2^a}(T_{k-a-1}(S_g))^i_jw_i\,|\nabla_g w|^{2a}\Bigr).
\end{align}
Here $\nabla$ denotes the Levi--Civita connection of $g$, $(T_{k-a}(S_g))^i_j=\frac{\partial \sigma_{k-a+1}}{\partial S_i^j}(g^{-1}S_g)$, and $(S_{g_{0}})_{p}^{j}=g^{jl}(S_{g_{0}})_{pl}$.
In particular, when $k=2$, \eqref{eq:intsigmak} holds on any Riemannian manifold.
\end{prop}

\begin{cor}\label{lowerboundofint}
Let $(M^{n},g_0)$ be locally conformally flat and let $g=e^{-2w}g_0\in \mathcal C_{k-1}$. If $k<\frac n2$, then
\[
\int_M \sigma_k(g)\,dv_g\ge -C\,\operatorname{Vol}(g)^{\frac{n-2k}{n}},
\]
where $C>0$ depends only on $n$, $k$, and $g_0$.
\end{cor}

\noindent\textit{Proof of Proposition~\ref{prop6.1}.}
By Corollary~\ref{lowerboundofint} and the Sobolev inequality on $\mathcal C_{k-1}$ proved in \cite{GWDuke}, we obtain
\eq{
\frac{\int_M \sigma_k(g)\,dv_g}{\left(\int_M \sigma_{k-1}(g)\,dv_g\right)^{\frac{n-2k}{n-2(k-1)}}}
=
\frac{\int_M \sigma_k(g)\,dv_g}{\operatorname{Vol}(g)^{\frac{n-2k}{n}}}
\left(\frac{\int_M \sigma_{k-1}(g)\,dv_g}{\operatorname{Vol}(g)^{\frac{n-2(k-1)}{n}}}\right)^{-\frac{n-2k}{n-2(k-1)}}
>-C.
}
\qed


\begin{proof}[Proof of Proposition~\ref{thm:intsigmak}]

Let $g=e^{-2w}g_{0}$, so that $g_{0}=e^{2w}g$.
Then
\begin{equation}\label{conformal change}
(S_{g})_{ij}=(S_{g_{0}})_{ij}+\nabla_{ij}w-w_{i}w_{j}+\frac{|\nabla_{g}w|^{2}}{2}g_{ij}.
\end{equation}
Here $\nabla$ is the Levi-Civita connection with respect to $g$.
In the following proof, we sometimes write $(T_{k-a})^{i}_{j}$ for $(T_{k-a}(S_g))^{i}_{j}$ when no confusion arises.

We begin with the following induction identity:
\begin{align}\label{eq:induction lemma-1}
 & \frac{a}{2^{a-1}}T_{k-a}(S_{g})(\nabla_{g}w,\nabla_{g}w)|\nabla_{g}w|^{2a-2}\\
= &-\frac{1}{2^a}\nabla^j(T_{k-a-1}(S_g)^i_jw_i|\nabla_g w|^{2a})\nonumber\\
& +\frac{a+1}{2^{a}}(T_{k-a-1})_{j}^{i}w^{j}w_{i}|\nabla_{g}w|^{2a}\nonumber\\
&+\frac{k+a}{2^{a}}\sigma_{k-a}(g)|\nabla_{g}w|^{2a} -\frac{(n-k+a+1)}{2^{a+1}}\sigma_{k-a-1}(g)|\nabla_{g}w|^{2a+2}\nonumber \\
 & -\frac{1}{2^{a}}(T_{k-a-1}(S_{g}))_{j}^{i}(S_{g_{0}})_{i}^{j}|\nabla_{g}w|^{2a}-\frac{a}{2^{a-1}}(T_{k-a-1}(S_{g}))_{j}^{i}(S_{g_{0}})_{k}^{j}w^{k}w_{i}|\nabla_{g}w|^{2a-2}.
\nonumber
\end{align}
To obtain \eqref{eq:induction lemma-1}, substituting $\sigma_{k-a}(g)\delta^i_{l}-(T_{k-a-1})_{j}^{i}(S_{g})_{l}^{j}$ into $(T_{k-a})^i_l$ and using \eqref{conformal change},
we have
\begin{align}
 & T_{k-a}(S_{g})(\nabla_{g}w,\nabla_{g}w)|\nabla_{g}w|^{2a-2} \nonumber \\
\overset{\eqref{conformal change}}{=} & \sigma_{k-a}(g)|\nabla_{g}w|^{2a}-(T_{k-a-1})_{j}^{i}w_{l}^{j}w^{l}w_{i}|\nabla_{g}w|^{2a-2}+\frac{1}{2}(T_{k-a-1})_{j}^{i}w^{j}w_{i}|\nabla_{g}w|^{2a}\label{eq:formula 1}\\
 & -(T_{k-a-1}(S_{g}))_{j}^{i}(S_{g_{0}})_{l}^{j}w^{l}w_{i}|\nabla_{g}w|^{2a-2}.\nonumber 
\end{align}
Since $\nabla_i (T_{k-a-1})^i_j=0$, we have
\begin{align*}
 & (T_{k-a-1}(S_{g}))_{j}^{i}w_{k}^{j}w^{k}w_{i}|\nabla_{g}w|^{2a-2}
= \frac{1}{2}(T_{k-a-1})_{j}^{i}\nabla^{j}(|\nabla_{g}w|^{2})w_{i}|\nabla_{g}w|^{2a-2}\\
= & -\frac{1}{2a}(T_{k-a-1})_{j}^{i}\nabla^{j}w_{i}|\nabla_{g}w|^{2a}+\frac{1}{2a}\nabla^j((T_{k-a-1})^i_j w_i|\nabla_g w|^{2a})\\
\overset{\eqref{conformal change}}{=} & -\frac{1}{2a}(T_{k-a-1})_{j}^{i}|\nabla_{g}w|^{2a}\left((S_{g})_{i}^{j}-(S_{g_{0}})_{i}^{j}+w^{j}w_{i}-\frac{|\nabla_{g}w|^{2}}{2}\delta_{i}^{j}\right)\\
&+\frac{1}{2a}\nabla^j((T_{k-a-1})^i_j w_i|\nabla_g w|^{2a})\\
= &\frac{1}{2a}\nabla^j((T_{k-a-1})^i_j w_i|\nabla_g w|^{2a})\\
& +\frac{(n-k+a+1)}{4a}\sigma_{k-a-1}(g)|\nabla_{g}w|^{2a+2}-\frac{k-a}{2a}\sigma_{k-a}(g)|\nabla_{g}w|^{2a}\\
 & -\frac{1}{2a}(T_{k-a-1})_{j}^{i}w^{j}w_{i}|\nabla_{g}w|^{2a}+\frac{1}{2a}(T_{k-a-1})_{j}^{i}(S_{g_{0}})_{i}^{j}|\nabla_{g}w|^{2a}.
\end{align*}
Thus, substituting this into (\ref{eq:formula 1}), we obtain
\begin{align}
 & T_{k-a}(S_{g})(\nabla_{g}w,\nabla_{g}w)|\nabla_{g}w|^{2a-2}\nonumber \\
= &-\frac{1}{2a}\nabla^j((T_{k-a-1})^i_j w_i|\nabla_g w|^{2a})+ \frac{a+1}{2a}(T_{k-a-1})_{j}^{i}w^{j}w_{i}|\nabla_{g}w|^{2a}\nonumber\\
&+\frac{k+a}{2a}\sigma_{k-a}(g)|\nabla_{g}w|^{2a}
 -\frac{(n-k+a+1)}{4a}\sigma_{k-a-1}(g)|\nabla_{g}w|^{2a+2}\nonumber \\
 & -\frac{1}{2a}(T_{k-a-1})_{j}^{i}(S_{g_{0}})_{i}^{j}|\nabla_{g}w|^{2a}-(T_{k-a-1}(S_{g}))_{j}^{i}(S_{g_{0}})_{l}^{j}w^{l}w_{i}|\nabla_{g}w|^{2a-2},\nonumber 
\end{align}
and thus, we deduce \eqref{eq:induction lemma-1}. 

By an induction argument, we further get that 
\begin{align*}
 & k\sigma_{k}(g^{-1}S_{g})=T_{k-1}(S_{g})^{ij}S_{g,ij}\\
\overset{\eqref{conformal change}}{=} & T_{k-1}(S_{g})^{ij}S_{g_{0},ij}-T_{k-1}(S_{g})^{ij}w_{i}w_{j}+(n-k+1)\sigma_{k-1}(g)\frac{|\nabla_{g}w|^{2}}{2}+\nabla_i(T_{k-1}(S_g)^{ij}\nabla_jw)\\
\overset{\eqref{eq:induction lemma-1}}{=}  
&\nabla_i(T_{k-1}(S_g)^{ij}\nabla_jw)+\frac12\nabla^j((T_{k-2}(S_g))^i_jw_i|\nabla_gw|^2)\\
& -(T_{k-2})_{j}^{i}w^{j}w_{i}|\nabla_{g}w|^{2}dv_{g}- \frac{k+1}{2}\sigma_{k-1}(g)|\nabla_{g}w|^{2}\\
 & +\frac{(n-k+2)}{4}\sigma_{k-2}(g)|\nabla_{g}w|^{4}+(n-k+1)\sigma_{k-1}(g)\frac{|\nabla_{g}w|^{2}}{2}\\
 & +T_{k-1}(S_{g})^{ij}S_{g_{0},ij} +\frac{1}{2}(T_{k-2})_{j}^{i}(S_{g_{0}})_{i}^{j}|\nabla_{g}w|^{2}+(T_{k-2})_{j}^{i}(S_{g_{0}})_{l}^{j}w^{l}w_{i}\\
\overset{\eqref{eq:induction lemma-1}}{=} 
&\nabla_i(T_{k-1}(S_g)^{ij}\nabla_jw)+\frac12\nabla^j((T_{k-2}(S_g))^i_jw_i|\nabla_gw|^2)+\frac{1}{2^2}\nabla^j((T_{k-3}(S_g))^i_jw_i|\nabla_g w|^4)\\
& -\frac{3}{4}(T_{k-3})_{j}^{i}w^{j}w_{i}|\nabla_{g}w|^{4}+\frac{(n-k+3)}{2^{3}}\sigma_{k-3}(S_{g})|\nabla_{g}w|^{6}\\
 & +(n-2k)\big(\frac{1}{2}\sigma_{k-1}(g)|\nabla_{g}w|^{2}+\frac{1}{2^{2}}\sigma_{k-2}(g)|\nabla_{g}w|^{4}\big)\\
 & +T_{k-1}(S_{g})^{ij}S_{g_{0},ij}+\frac{1}{2}(T_{k-2})_{j}^{i}(S_{g_{0}})_{i}^{j}|\nabla_{g}w|^{2}+\frac{1}{2^{2}}(T_{k-3})_{j}^{i}(S_{g_{0}})_{i}^{j}|\nabla_{g}w|^{4}\\
 & +(T_{k-2})_{j}^{i}(S_{g_{0}})_{l}^{j}w^{l}w_{i}+(T_{k-3})_{j}^{i}(S_{g_{0}})_{l}^{j}w^{l}w_{i}|\nabla_{g}w|^{2}\\
= & (n-2k)\sum_{j=1}^{k}\frac{\sigma_{k-j}\left(g^{-1}\circ S_{g}\right)}{2^{j}}|\nabla w|_{g}^{2j}+\sum_{a=1}^{k}\frac{1}{2^{a-1}}(T_{k-a}(S_{g}))_{j}^{i}(S_{g_{0}})_{i}^{j}|\nabla_{g}w|^{2(a-1)}\\
 & +\sum_{a=2}^{k}\frac{a-1}{2^{a-2}}(T_{k-a}(S_{g}))_{j}^{i}(S_{g_{0}})_{l}^{j}w^{l}w_{i}|\nabla_{g}w|^{2(a-2)}+\sum_{a=0}^{k-1}\nabla^j(\frac{1}{2^a}(T_{k-a-1}(S_g))^i_jw_i|\nabla_g w|^{2a}).
\end{align*}

\end{proof}

\begin{proof}[Proof of Corollary \ref{lowerboundofint}]
Since $g\in \mathcal C_{k-1}$, the only possible negative terms in  \eqref{eq:intsigmak} are  
$$\int_M(T_{k-1}(S_{g}))_{j}^{i}(S_{g_{0}})_{i}^{j}dv_g$$
and 
$$\sum_{a=2}^{k}\frac{a-1}{2^{a-2}}{\int_{M}(T_{k-a}(S_{g}))_{j}^{i}(S_{g_{0}})_{k}^{j}w^{k}w_{i}|\nabla_{g}w|^{2(a-2)}}dv_g.$$
In the remaining proof, we will deal with these two types of terms.

Recalling that $\sigma_{1}(g_{0})=\sigma_{1}(g_{0}^{-1}S_{g_0})$, we have
\begin{align}
 & \int(T_{k-1}(S_{g}))_{j}^{i}(S_{g_{0}})_{i}^{j}dv_g\nonumber \\
= & \int\sigma_{k-1}(g)\sigma_{1}(g_{0})e^{2w}dv_g-\int(T_{k-2})_{l}^{i}(S_{g})_{j}^{l}(S_{g_{0}})_{i}^{j}\nonumber \\
\overset{\eqref{conformal change}}{=} & \int\sigma_{k-1}(g)\sigma_{1}(g_{0})e^{2w}dv_g-\int(T_{k-2})_{l}^{i}(S_{g_{0}})_{j}^{l}(S_{g_{0}})_{i}^{j}\nonumber \\
 & -\int(T_{k-2})_{l}^{i}\nabla_{j}^{l}w(S_{g_{0}})_{i}^{j}+\int(T_{k-2})_{l}^{i}w^{l}w_{j}(S_{g_{0}})_{i}^{j}-\int(T_{k-2})_{j}^{i}(S_{g_{0}})_{i}^{j}\frac{|\nabla_{g}w|^{2}}{2}\nonumber \\
= & \int\sigma_{k-1}(g)\sigma_{1}(g_{0})e^{2w}dv_g-\int(T_{k-2})_{l}^{i}(S_{g_{0}})_{j}^{l}(S_{g_{0}})_{i}^{j}+\int(T_{k-2})_{l}^{i}\nabla_{j}w\nabla^{l}(S_{g_{0}})_{i}^{j}\label{eq:0}\\
 & +\int(T_{k-2})_{l}^{i}w^{l}w_{j}(S_{g_{0}})_{i}^{j}-\int(T_{k-2})_{j}^{i}(S_{g_{0}})_{i}^{j}\frac{|\nabla_{g}w|^{2}}{2},\nonumber 
\end{align}
where in the last equality we have used $\nabla^l (T_{k-2})^i_l=0$.
Since $T_{k-2}(S_{g})^{i}_{j}$ is positive (because $g\in \mathcal C_{k-1}$), we have $|T_{k-2}(S_{g})^{i}_{j}|\le C\sigma_{k-2}(g)$ for some constant $C>0$.

We now estimate the terms in \eqref{eq:0} one by one.
\begin{align}
 & \int_{M}(T_{k-2})_{l}^{i}(S_{g_{0}})_{j}^{l}(S_{g_{0}})_{i}^{j}\,dv_{g}
\le  C\,\max|S_{g_{0}}|_{g_0}^{2}\int_M e^{4w}\sigma_{k-2}(g)\,dv_{g}.
\label{eq:1}
\end{align}
Next,
\begin{align}
 \int_{M}(T_{k-2})_{j}^{i}(S_{g_{0}})_{i}^{j}\frac{|\nabla_{g}w|^{2}}{2}\,dv_{g}
\le {}& C\int_{M}\sigma_{k-2}(g)|S_{g_{0}}|_{g_{0}}e^{2w}\frac{|\nabla_{g}w|^{2}}{2}\,dv_g\nonumber \\
\le {}& C\max|S_{g_{0}}|_{g_{0}}\left(\int_{M}\sigma_{k-2}(g)|\nabla_{g}w|^{4}\,dv_g\right)^{\frac{1}{2}}\left(\int_M\sigma_{k-2}(g)e^{4w}\,dv_g\right)^{\frac{1}{2}}\nonumber \\
\le {}& \varepsilon\int_{M}\sigma_{k-2}(g)|\nabla_{g}w|^{4}\,dv_g+C\int_M\sigma_{k-2}(g)e^{4w}\,dv_g.
\label{eq:2}
\end{align}
Moreover,
\begin{align}
\bigg|\int_M (T_{k-2})_{l}^{i}w^{l}w_{j}(S_{g_{0}})_{i}^{j}\,dv_g \bigg|
&\le\varepsilon\int_{M}\sigma_{k-2}(g)|\nabla_{g}w|^{4}\,dv_g+C\int_M\sigma_{k-2}(g)e^{4w}\,dv_g
\label{eq:3}
\end{align}
for any $\varepsilon>0$.

Also, by $|\nabla_{g_0}w|e^w=|\nabla_g w|$, we have 
\begin{align}
\bigg|\int_M (T_{k-2})_{l}^{i}\nabla_{j}w\,\nabla^{l}(S_{g_{0}})_{i}^{j}\,dv_g \bigg|
&\le C\max|\nabla_{g_{0}}S_{g_{0}}|\int_M\sigma_{k-2}(g)|\nabla_{g_{0}}w|e^{4w}\,dv_g\nonumber\\
&\quad +C\int_M \sigma_{k-2}(g)|\nabla_g w|^2 e^{2w} dv_g
\nonumber \\
&\le\varepsilon\int_{M}\sigma_{k-2}(g)|\nabla_{g}w|^{4}\,dv_g+C\int_M\sigma_{k-2}(g)e^{4w}\,dv_g.
\label{eq:4}
\end{align}

For any $a=2,\dots,k-1$,
\begin{align}
 & \bigg|\int_{M}(T_{k-a}(S_{g}))_{j}^{i}(S_{g_{0}})_{k}^{j}w^{k}w_{i}|\nabla_{g}w|^{2(a-2)}\,dv_g\bigg|
\nonumber \\
\le {}& C\max|S_{g_{0}}|_{g_{0}}\int_{M}\sigma_{k-a}(g)|\nabla_{g}w|^{2(a-1)}e^{2w}\,dv_{g}
\nonumber \\
\le {}& \varepsilon_{1}\int_{M}\sigma_{k-a}(g)|\nabla_{g}w|^{2a}\,dv_{g}+C_{\varepsilon_{1}}\int_{M}\sigma_{k-a}(g)e^{2aw}\,dv_{g}.
\label{eq:5}
\end{align}
for any $\varepsilon_1>0$.

Since $\nabla^{j}(T_{k-a-1}(S_{g}))_{j}^{i}=0$, we have 
\begin{align}
 (k-a)\int_{M}\sigma_{k-a}(g)e^{2aw}\,dv_{g}
&= \int_M T_{k-a-1}(S_{g})_{j}^{i}\left((S_{g_{0}})_{i}^{j}+\nabla_{i}^{j}w-w_{i}w^{j}+\frac{|\nabla_{g}w|^{2}}{2}\delta_{i}^{j}\right)e^{2aw}\,dv_{g}
\nonumber \\
&\le \varepsilon_{a}\int_{M}\sigma_{k-a-1}(g)|\nabla_{g}w|^{2(a+1)}\,dv_{g}+C_{\varepsilon_{a}}\int_{M}\sigma_{k-a-1}(g)e^{2(a+1)w}\,dv_{g}.
\label{eq:induction formula}
\end{align}
for any $\varepsilon_{a}>0$.
Thus, by \eqref{eq:induction formula} (iterated), we obtain
\begin{equation}
\int_{M}\sigma_{k-a}(g)e^{2aw}\,dv_{g}
\le\varepsilon_{0}\sum_{i=a}^{k-1}\int_{M}\sigma_{k-i-1}(g)|\nabla_{g}w|^{2(i+1)}\,dv_{g}+C\int_{M}e^{2kw}\,dv_{g}.
\label{eq:6}
\end{equation}
In particular,
\begin{equation}
\int_{M}\sigma_{k-2}(g)e^{4w}\,dv_{g}
\le\varepsilon_{0}\sum_{i=2}^{k-1}\int_M\sigma_{k-i-1}(g)|\nabla_{g}w|^{2(i+1)}\,dv_{g}+C\int_{M}e^{2kw}\,dv_{g}.
\label{eq:7}
\end{equation}
Thus, substituting  (\ref{eq:1})-(\ref{eq:4}) and (\ref{eq:7}) into \eqref{eq:0}, we find that for $\varepsilon>0$ and $\varepsilon_0>0$,
\[
\int_M(T_{k-1}(S_{g}))_{j}^{i}(S_{g_{0}})_{i}^{j}\ge-(3\varepsilon+C\varepsilon_{0})\sum_{a=1}^{k-1}\int\sigma_{k-a-1}(g)|\nabla_{g}w|^{2(a+1)}dv_{g}-C_{\varepsilon,\varepsilon_0}\int_{M}e^{2kw}\,dv_{g}.
\]
By \eqref{eq:5}--\eqref{eq:7}, for each $2\le a\le k$ we obtain the estimate
\begin{align*}
\biggl|\int_{M}(T_{k-a}(S_{g}))_{j}^{i}(S_{g_{0}})_{k}^{j}\,w^{k}w_{i}\,|\nabla_{g}w|^{2(a-2)}\,dv_g\biggr|
&\le \varepsilon_{0}\sum_{i=a-1}^{k-1}\int_{M}\sigma_{k-i-1}(S_{g})\,|\nabla_{g}w|^{2(i+1)}\,dv_{g}
+ C\int_{M}e^{2kw}\,dv_{g}.
\end{align*}
Combining this with \eqref{eq:intsigmak} and choosing $\varepsilon_0,\varepsilon>0$ sufficiently small, we conclude that if $k<\frac n2$ then
\begin{equation}\label{can be used to get the information of eigenvalue}
k\int_{M}\sigma_{k}(g^{-1}S_{g})\,dv_{g}
\ge {} \frac{n-2k}{2}\sum_{j=1}^{k}\int_{M}\frac{\sigma_{k-j}\!
\left(g^{-1} S_{g}\right)}{2^{j}}\,|\nabla w|_{g}^{2j}\,dv_g
- C\int_{M}e^{2kw}\,dv_{g}.
\end{equation}
Here we used that, for $a\ge2$,
\[
\int_{M}(T_{k-a}(S_{g}))_{j}^{i}(S_{g_{0}})_{i}^{j}\,|\nabla w|_{g}^{2(a-1)}\,dv_g\ge0,
\]
which follows from G\aa rding's inequality.

If $2k<n$, then
\begin{equation}\label{eq:8}
\int_{M}e^{2kw}\,dv_{g}
\le\left(\int_{M}e^{-nw}\,dv_{g_{0}}\right)^{\frac{n-2k}{n}}\left(\int_{M}1\,dv_{g_{0}}\right)^{\frac{2k}{n}}
\le C\,\operatorname{Vol}(g)^{\frac{n-2k}{n}}.
\end{equation}
Therefore, we prove this Corollary by \eqref{can be used to get the information of eigenvalue}.
\end{proof}

\section{Appendix. Examples}
\label{AppendixA}
In this appendix, we compute explicit examples without  positivity of the scalar curvature.

Let
\[
g_0=g_{\mathbb{S}^{n}}=\frac{1}{1-s^{2}}\,ds^{2}+\left(1-s^{2}\right)g_{\mathbb{S}^{n-1}},
 \qquad s=x_{n+1}.
\]
For a parameter $\ell>0$, consider the conformal metric
\[
 g_{\ell}=e^{-2\ell s^{2}}g_{\mathbb{S}^{n}}.
\]
We recall the identities
\[
|\nabla_{g_{0}}s|^{2}=1-s^{2},
\qquad
\nabla_{g_{0}}^{2}s=-s\,g_{0}.
\]
For a conformal change $g_{u}=e^{-2u}g_{0}$, the Schouten tensor transforms as
\[
S_{g_{u}}=S_{g_{0}}+\nabla^{2}u+du\otimes du-\frac{1}{2}|\nabla_{g_{0}}u|^{2}g_{0}.
\]

Set
\[
 u_{\ell}(s):=\ell s^{2},
 \qquad
 g_{\ell}:=e^{-2u_{\ell}}g_{0}=e^{-2\ell s^{2}}g_{0}.
\]
Then
\[
\nabla_{g_{0}}u_{\ell}=u_{\ell}'\,\nabla_{g_{0}}s,
\qquad
|\nabla_{g_{0}}u_{\ell}|^{2}=(u_{\ell}')^{2}|\nabla_{g_{0}}s|^{2}=4\ell^{2}s^{2}(1-s^{2}),
\]
and
\[
\nabla_{g_{0}}^{2}u_{\ell}=u_{\ell}''\,\nabla_{g_{0}}s\otimes\nabla_{g_{0}}s-s\,u_{\ell}'\,g_{0}.
\]
The eigenvalues of $g_{0}^{-1}S_{g_{\ell}}$ are
\[
\begin{gathered}
\lambda_{1}(s)=\frac{1}{2}+2\ell-4\ell s^{2}+2\ell^{2}s^{2}\left(1-s^{2}\right),\\
\lambda_{2}(s)=\frac{1}{2}-2\ell s^{2}-2\ell^{2}s^{2}\left(1-s^{2}\right),
\end{gathered}
\]
where $\lambda_{2}(s)$ has multiplicity $n-1$.
Hence
\[\s_1(g_{0}^{-1}S_{g_{\ell}})=\frac{n}{2}+2\ell -2\ell (n+1)s^2+(4-2n)\ell ^2s^2(1-s^2)\]
and $\s_1 $ is not everywhere positive, when $\ell$ is big.
\[
\sigma_{2}(g_{0}^{-1}S_{g_{\ell}})=\frac{(n-1)(n-2)}{2}\,\lambda_{2}(s)^{2}+(n-1)\,\lambda_{1}(s)\lambda_{2}(s).
\]

Since
\[
dv_{g_0}=dv_{g_{\mathbb{S}^{n}}}=\left(1-s^{2}\right)^{\frac{n-2}{2}}dv_{g_{\mathbb{S}^{n-1}}}\,ds,
\]
we have
\begin{align*}
&\int_{\mathbb{S}^{n}}\sigma_{2}(g_\ell)\,dv_{g_{\ell}}\\
&=\omega_{n-1}\int_{-1}^{1}e^{-(n-4)\ell s^{2}}
\left(\frac{(n-1)(n-2)}{2}\,\lambda_{2}(s)^{2}+(n-1)\,\lambda_{1}(s)\lambda_{2}(s)\right)
\left(1-s^{2}\right)^{\frac{n-2}{2}}ds\\
&=2\omega_{n-1}\int_{0}^{1}e^{-(n-4)\ell s^{2}}\left(1-s^{2}\right)^{\frac{n-2}{2}}
\left(\frac{1}{2}-2\ell s^{2}-2\ell^{2}s^{2}\left(1-s^{2}\right)\right)\\
&\quad\times\Biggl((n-1)\left(\frac{1}{2}+2\ell-4\ell s^{2}+2\ell^{2}s^{2}\left(1-s^{2}\right)\right)
+\frac{(n-1)(n-2)}{2}\left(\frac{1}{2}-2\ell s^{2}-2\ell^{2}s^{2}\left(1-s^{2}\right)\right)\Biggr)ds.
\end{align*}
And
\begin{align*}
   \frac{1}{2(n-1)} \int_{\mathbb{S}^n}R_{g_\ell } dv_{g_\ell }
    =&2\omega_{n-1}\int_0^1
e^{-(n-2)\ell s^2}(1-s^2)^{\frac{n-2}{2}}\\
&\times\left(\frac{n}{2}+2\ell -2\ell (n+1)s^2+(4-2n)\ell ^2s^2(1-s^2)\right)ds.
\end{align*}

When $n\ge 5$, the examples were discussed in \cite{GWW2} and we showed that 
\[\frac{\int_{\mathbb{S}^{n}}\sigma_{2}(g_\ell )dv_{g_{\ell }}}{Vol(g_\ell )^{\frac{n-4}{n}}}\rightarrow -\infty, \quad \text{as}\quad  \ell \rightarrow +\infty,\]
and 
\[\frac{\int_{\mathbb{S}^{n}}\sigma_{2}(g_\ell )dv_{g_{\ell }}}{(\int_{\mathbb{S}^n}R_
{g_\ell } dv_{g_\ell})^{\frac{n-4}{n-2}}}\rightarrow -\infty, \quad \text{as}\quad  \ell \rightarrow +\infty.\]
Hence the corresponding infimum for both functionals in the whole class $[g_0]$ is $-\infty$.

Here we  check the case $n=3$. 
From above, we can compute
\begin{equation}\label{3-dim sigma2}
\int_{\mathbb{S}^{3}}\sigma_{2}\left(g_{\ell}^{-1}S_{g_{\ell}}\right)dv_{g_{\ell}}=Vol(g_{\mathbb S^2})\frac{9}{2}\sqrt{\pi} e^\ell \sqrt{\ell}+o(e^\ell\sqrt{\ell}),
\end{equation}
and \begin{align}\label{3-dim integral of sigma1}
&\frac{1}{4}\int_{\mathbb{S}^{3}} R_{g_{\ell}} dv_{g_{\ell}}
=Vol(g_{\mathbb S^2})\sqrt{\pi} \sqrt{\ell}+o(\sqrt \ell),
\end{align}
and 
\begin{align}
Vol(g_{\ell})
=Vol(g_{\mathbb S^2})\frac{1}{\sqrt{\ell}}\sqrt{\frac{\pi}{3}}+o(\frac{1}{\sqrt \ell}).
\end{align}
In particular, as $\ell\to\infty$,
$$Vol(g_{\ell})^{1/3}\int_{\mathbb{S}^{3}}\sigma_2(g_{\ell})dv_{g_{\ell}}\to \infty, \quad
\int_{\mathbb{S}^{3}} \sigma_1(g_{\ell}) dv_{g_{\ell}}\int_{\mathbb{S}^{3}} \sigma_2(g_{\ell})dv_{g_{\ell}}\to \infty.$$
Hence both inequalities in Corollary \ref{k=2andn=3} fail without the positivity of the scalar curvature.  
Moreover, as $\ell\to\infty$,
$$\frac{Vol(g_{\ell})\int_{\mathbb{S}^{3}} \sigma_2(g_{\ell})dv_{g_{\ell}}}{(\int_{\mathbb{S}^{3}} \sigma_1(g_{\ell}) dv_{g_{\ell}})^2}\to \infty.$$
Hence, Conjecture~1 in \cite{GWAdv} does not hold in general.

\medskip

\noindent{\it Acknowledgments.}
This work was carried out while W.~Wei was visiting the University of Freiburg, supported by an Alexander von Humboldt research fellowship. She thanks the Institute of Mathematics at the University of Freiburg for its hospitality. She is also partially supported by NSFC (Grant No.~12571218 and 12271244). W.~Wei is grateful to Professor Jingang Xiong for asking, during her talk at Beijing Normal University in January 2026, whether the condition \(g\in\mathcal{C}_{k-2}\) suffices to ensure the positivity of \(\int_M \sigma_k(g)dv_g\). This question motivated the construction of the counterexamples presented here.
\bibliography{bibforSigmakSobolev}
\bibliographystyle{amsplain}

\end{document}